\documentclass[12pt]{amsart}
\usepackage{lipsum}
\usepackage{amssymb}
\usepackage{enumerate}
\usepackage{amsthm}
\usepackage{tikz-cd}
\usepackage{tcolorbox}
\usepackage{amsmath}
\usepackage[mathscr]{euscript}
\usepackage[utf8]{inputenc}
\usepackage[english]{babel}
\usepackage{hyperref}
\hypersetup{
    colorlinks=true,
    linkcolor=blue,
    citecolor=red,
    filecolor=red,
    urlcolor=violet,
}

\makeatletter
\@namedef{subjclassname@2020}{%
  \textup{2020} Mathematics Subject Classification}
\makeatother
\newtheorem{thm}{Theorem}[section]
\newtheorem{lem}[thm]{Lemma}
\newtheorem{prop}[thm]{Proposition}
\newtheorem{defi}[thm]{Definition}

\newtheorem{corl}[thm]{Corollary}
\newtheorem{xrem}[thm]{Remark}
\newtheorem{exm}[thm]{Example}
\newtheorem{conj}[thm]{Conjecture}

\DeclareMathOperator{\rank}{{rank}}

\DeclareMathOperator{\HN}{{HN}}
\DeclareMathOperator{\gr}{{gr}}

\DeclareMathOperator{\Gr}{{Gr}}
\DeclareMathOperator{\mult}{{mult}}
\DeclareMathOperator{\Sym}{{Sym}}
\DeclareMathOperator{\Nef}{{Nef}}
\DeclareMathOperator{\NE}{{NE}}

\begin{document}
\baselineskip=16pt

\date{\today}
\subjclass[2020]{14J60, 14H60}
\keywords{Higgs bundles, Seshadri constants, ampleness, numerical effectiveness}
\author{Krishna Hanumanthu}
\author{Snehajit Misra}
\author{Nabanita Ray}

\address{Chennai Mathematical Institute, H1 SIPCOT IT Park, Siruseri, Kelambakkam 603103, India}
\email[Krishna Hanumanthu]{krishna@cmi.ac.in}
\address{Department of Mathematics and Computing, Indian Institute of Technology (Indian School of Mines) Dhanbad, Jharkhand - 826004, India}
\email[Snehajit Misra]{snehajitm@iitism.ac.in}
\address{Indraprastha Institute of Information Technology, Delhi, Okhla Industrial Estate, Phase III,
New Delhi, Delhi 110020, India}
\email[Nabanita Ray]{nabanita@iiitd.ac.in}

\begin{abstract}
We define Seshadri constants for Higgs bundles on smooth projective varieties over algebraically closed fields of characteristic zero. This definition is inspired by and  analogous to the notion of Seshadri constants for ordinary vector bundles. We prove a series of properties of Higgs Seshadri constants which are analogous to the corresponding properties in the case of ordinary Seshadri constants. In particular, we prove a Seshadri criterion for Higgs ampleness and prove that Higgs Seshadri constants can be computed by restriction to curves. 
\end{abstract}

\title{Seshadri constants of Higgs Vector bundles}
\maketitle

\vskip 4mm

\section{Introduction}\label{intro}
Seshadri constants were defined by Demailly \cite{Dem90} as local invariants to study the separation of jets of line bundles on complex projective varieties. Demailly's motivation came from an ampleness criterion due to Seshadri \cite{Har70} and the immediate purpose was to study the Fujita Conjecture about positivity of adjoint line bundles. This notion has been extensively studied, and Seshadri constants have turned out to have fundamental significance to the geometry of algebraic varieties, especially in connection with positivity aspects. As part of this extensive study, they have been defined for vector bundles of arbitrary rank by Hacon \cite{Hac00} and, more recently,  for vector bundles with additional structure such as parabolic bundles \cite{BHMR23}. 
In this paper, we continue this study by defining Seshadri constants  for another important case, that of Higgs vector bundles. 

%1) Briefly recall the definition of Seshadri constants and basic properties and references. 

We recall below the classical notion of Seshadri constants of line bundles introduced by Demailly. 

\begin{defi} [Seshadri constant of a line bundle at a point]
		Let $X$ be a projective variety and let $L$ be a nef line bundle on $X$. Then for  $x \in X$, the real number
		\begin{equation*}
			\varepsilon(L;x) := \inf\limits_{x\in C \subseteq X} \Bigl\{ \frac{L \cdot C}{\rm{mult}_x C} \Bigr\}
		\end{equation*}
		is called the \textit{Seshadri constant} of $L$ at $x$, where the infimum is taken over all irreducible and reduced curves $C$ such that $x \in C$.
	\end{defi}

The Seshadri criterion for ampleness says that a nef line bundle $L$ is ample if and only if 
$\inf\limits_{x\in X} \varepsilon(L;x) > 0$.  Seshadri constants of ample line bundles can be arbitrarily small positive numbers, as shown by an example of Miranda; see \cite{EL93}.

Seshadri constants of line bundles have been studied extensively, especially on algebraic surfaces. For a survey, see \cite{primer}.

%2) Briefly recall the vector bundle case. 

Since the notions of nefness and ampleness generalize to vector bundles of arbitrary rank, it is natural to ask whether Seshadri constants can be defined in more generality. 
This was originally done by Hacon \cite{Hac00} and later expanded by Fulger and Murayama \cite{FM21} who generalized the notion to include the relative case. We recall their 
definition below. 

\begin{defi}\rm
 Let $\rho : Y\longrightarrow X$ be a morphism of projective schemes defined over an algebraically closed field and let $x\in X$ be a closed point in $X$. Let $\mathcal{C}_{\rho,x}$ be the set of all irreducible curves on $Y$ that intersect the fiber $\rho^{-1}(x)$, but are not contained in the support of $\rho^{-1}(x).$
 %We say a line bundle $\xi$ is $\rho$-nef (respectively, $\rho$-ample) if $\xi\vert_{\rho^{-1}(x)}$ is nef (respectively, ample) for all $x\in X$.

    The Seshadri constant of a nef line bundle $\xi$ on $Y$ at a point $x$ in $X$ is defined as
    $$\varepsilon(\xi\, ;\, x) = \inf_{C\in \mathcal{C}_{\rho,x}}\Bigl\{ \dfrac{\xi\cdot C}{\mult_x\rho_*C}\Bigr\},$$
\end{defi}
where the infimum is taken over all irreducible curves $C\in \mathcal{C}_{\rho,x}.$

If $E$ is a vector bundle on $X$, then the Seshadri constant $\varepsilon(E;\, x)$ of $E$ at $x \in X$ is defined as $\varepsilon(\xi\, ;\, x)$, where 
$\xi =\mathcal{O}_Y(1)$
is the Serre line bundle 
on the projective bundle
$Y=\mathbb{P}_X(E)$
over $X$ associated to $E$.

Although Seshadri constants of vector bundles are an active topic of study, much less is known about them in comparison with the line bundle case. See \cite{BHM24} for some recent results. 

%3) Describe the contents of the paper. 

In this paper, we continue this study by extending the notion of Seshadri constants to the case of \textit{Higgs} vector bundles.  A Higgs bundle on a smooth complex projective variety  is a natural generalization of a vector bundle, equipped with  additional structure called a \textit{Higgs field}. Higgs vector bundles were first introduced by Hitchin in connection with the self-duality equations on a Riemann surface (see \cite{Hi87a} and \cite{Hi87b}).   Subsequently, Simpson extended this notion to higher-dimensional varieties via his investigation of variations of Hodge structures (see \cite{S88}). Simpson extended the results of Hitchin to Higgs bundles over higher-dimensional complex projective manifolds. A Higgs bundle comes equipped with a natural stability condition (see Definition \ref{defn1}) which allows one to study the moduli spaces of  Higgs bundles. The study of  Higgs bundles is central in
variations of Hodge structure, the study of degenerations, non-abelian Hodge correspondence 
and in the geometric Langlands program.

In Section \ref{prelims}, we recall the definition and important properties of Higgs bundles. The material in this section is well-known, but we include relevant results for the convenience of the reader. 

In Section \ref{Higgs-ample}, we recall the notions of nefness and ampleness which were introduced earlier in the literature. Most (but not all) results in this section are known and we include them only for completeness. 

In Section \ref{sc}, we introduce Seshadri constants of Higgs bundles, which we call \textit{Higgs Seshadri constants}. We prove several results about them analogous to the usual Seshadri constants of vector bundles. In particular, we prove a Seshadri criterion for Higgs ampleness (Theorem \ref{ses-criteria}). 
In Theorem \ref{thm-curves},
we show that the Higgs Seshadri constant of a Higgs bundle on smooth curves is equal to the minimum Higgs slope of the Higgs bundle. This is analogous to a result proved by Hacon in  the usual vector bundle case. 
Again, analogously to the usual vector bundle case, we show that the Higgs Seshadri constants can be computed by restriction to curves (Theorem \ref{thm-res-curve}). 

In Section \ref{properties}, we continue 
the study of Higgs Seshadri constants by proving several more results about them. We compute the Higgs Seshadri constants on ruled surfaces in Corollary \ref{5.5} and Proposition \ref{5.7}. Among other results, 
we study the behaviour of Higgs Seshadri constants for tensor products (Theorem \ref{tensor}), symmetric powers (Theorem \ref{symmetric}) and their behaviour in exact sequences (Corollary \ref{exact}). We also prove a Miranda type result (Theorem \ref{miranda-exm}) which shows that Higgs Seshadri constants can be arbitrarily small.

We also give many examples, comparing the usual and Higgs Seshadri constants and illustrating the results we prove about Higgs Seshadri constants. 

We work throughout over an algebraically closed field of characteristic zero.

\section{Preliminaries}\label{prelims}
Let $X$ be a smooth projective variety defined over an algebraically closed field 
%$\mathbb{K}$ 
of characteristic zero.  A \textit{Higgs sheaf} on $X$ is a pair $\mathcal{E} = (E,\theta)$, where $E$ is a coherent sheaf on $X$ and $\theta : E\longrightarrow E\otimes \Omega^1_X$ is an $\mathcal{O}_X$-module morphism such that the composition 
$$\theta \wedge \theta  : E \xrightarrow{\theta} E \otimes \Omega^1_X\xrightarrow{\theta\otimes id} E\otimes \Omega^1_X \otimes \Omega_X^1 \rightarrow E\otimes \Omega^2_X $$
is zero. The $\mathcal{O}_X$-module map $\theta$ is called the \textit{Higgs field} of the Higgs sheaf $\mathcal{E}$. A Higgs sheaf $\mathcal{E} = (E,\theta)$ is called a \textit{Higgs bundle} if $E$ is locally free. For a map $\phi : Y \longrightarrow X$ between two smooth projective varieties $X$ and $Y$, and a Higgs bundle $\mathcal{E}=(E,\theta)$ on $X$, its pullback $\phi^*\mathcal{E}$ under $\phi$ is defined as the Higgs bundle
 $(\phi^*E,\phi^*\theta)$ on $Y$, where $\phi^*\theta$ is the composition map
 \begin{align*}
  \phi^*\theta :\phi^*E\longrightarrow \phi^*E \otimes \phi^*\Omega^1_X \longrightarrow \phi^*E\otimes \Omega^1_Y.
 \end{align*}
 
  If $\mathcal{E}=(E,\theta)$ and $\mathcal{G}=(G,\phi)$ are Higgs sheaves on $X$, 
   a morphism $f:(E,\theta)\longrightarrow(G,\phi)$ is a morphism of $\mathcal{O}_X$-modules $f:E\longrightarrow G$ such that the following diagram commutes:
\begin{center}
 \begin{tikzcd}
E \arrow[r, "f"] \arrow[d, "\theta"]
& G \arrow[d, "\phi" ] \\
E \otimes\Omega^1_X \arrow[r, "f\otimes id" ]
& G\otimes\Omega^1_X
\end{tikzcd}
\end{center}

The Higgs sheaves $\mathcal{E}$ and $\mathcal{G}$ are said to isomorphic if there is a morphism $f : \mathcal{E} \longrightarrow \mathcal{G}$ such that $f$ is an isomorphism
as an $\mathcal{O}_X$-module morphism.

For two Higgs bundles $\mathcal{E} = (E,\theta)$ and $\mathcal{F}=(F,\phi)$ on $X$,  their tensor product is also a Higgs bundle defined in the following way: we define the map $$\theta\otimes \psi : E\otimes F \longrightarrow (E\otimes F)\otimes \Omega_X^1$$ as follows:  for any local sections $u$ and $v$ of $E$ and $F$ respectively, set $$(\theta\otimes \phi)(u\otimes v) = \theta(u)\otimes v + u \otimes \phi(v).$$

In particular, $\mathcal{E}\otimes L$ is a Higgs bundle for any line bundle $L$ by considering the trivial Higgs field 
on the line bundle $L$.

For any integer $k > 1$,  $$ \otimes^k \theta : E^{\otimes k} \longrightarrow E^{\otimes k} \otimes \Omega^1_X$$
is defined as $\otimes^k \theta = \theta \otimes \theta \otimes \cdots \otimes \theta $, iterated $k$-times.  The resulting Higgs bundle is denoted by $\mathcal{E}^{\otimes k}$.  A Higgs field $\theta$ is \it nilpotent \rm if 
 $\otimes^k \theta = 0$ for some $k$.

The Higgs field  $\otimes^k \theta $ induces the morphisms  $$\Sym^k\theta : \Sym^kE \longrightarrow \Sym^kE \otimes \Omega^1_X.$$ This gives a Higgs field on the symmetric power $\Sym^kE$, and the resulting Higgs bundle  is denoted by $\Sym^k\mathcal{E} = (\Sym^kE, \Sym^k\theta).$

\subsection{Higgs Grassmann scheme}
Let $\mathcal{E} = (E,\theta)$ be a Higgs bundle of rank $r$ on $X$ (i.e., $E$ is a vector bundle of rank $r$) and let $1\leq k \leq r-1$ be an integer. Let $p_k: \text{Gr}_k(E) \longrightarrow X$ be the Grassmann bundle parametrizing locally free quotients of $E$ of rank $k$. We consider the following exact sequence:
\begin{align}\label{seq-Grassmann}
 0 \longrightarrow S_{E,r-k} \xrightarrow{\psi} p_k^*E \xrightarrow{\eta}Q_{E,k}\longrightarrow 0,
\end{align}
where  $S_{E,r-k}$ is the universal rank $r-k$ subbundle of $p_k^*E$ and $Q_{E,k}$ is the universal rank $k$ quotient of rank $k$. 

The \textit{Higgs Grassmann scheme}
corresponding to the Higgs bundle $\mathcal{E}=(E,\theta)$, denoted $\mathcal{G}r_k(\mathcal{E})$, is a 
closed subscheme of $\text{Gr}_k(E)$
defined as the zero loci in $\text{Gr}_k(E)$ of the composite morphism
%defined the closed subscheme $\mathcal{G}r_k(\mathcal{E}) \subseteq \text{Gr}_k(E)$ (called as \it Higgs Grassmann schemes\rm) is defined as the zero loci of the composite morphisms 
$$\bigl(\eta \circ \text{Id}_{\Omega^1_X}\bigr) \circ p_k^*(\theta)\circ \psi : S_{E,r-k} \longrightarrow Q_{E,k} \otimes p_k^*\Omega^1_X.$$

Restricting  (\ref{seq-Grassmann}) to the Higgs Grassmann scheme $\mathcal{G}r_k(\mathcal{E})$, we have the following universal short exact sequence
\begin{align}\label{seq1}
 0 \longrightarrow \mathcal{S}_{\mathcal{E},r-k} \xrightarrow{\psi} \rho_k^*\mathcal{E} \xrightarrow{\eta}\mathcal{Q}_{\mathcal{E},k}\longrightarrow 0.
\end{align}
Here $\rho_k = p_k\vert_{\mathcal{G}r_k(\mathcal{E})}$, and the Higgs vector bundle $\mathcal{Q}_{\mathcal{E},k} = Q_{E,k}\vert_{\mathcal{G}r_k(\mathcal{E})}$ is equipped with the quotient Higgs field induced by $\rho_k^*\theta.$ The Higgs Grassmann scheme 
$\mathcal{G}r_k(\mathcal{E})$ has the following universal properties:
\vspace{2mm}

    For any morphism $f:Y\longrightarrow X$ and any rank $k$ locally free Higgs quotient $\mathcal{F}$ of $f^*\mathcal{E}$, there is  a morphism $\psi_k:Y\longrightarrow \mathcal{G}r_k(\mathcal{E})$ over $X$ such that $\mathcal{F} = \psi_k^*\mathcal{Q}_{\mathcal{E},k}.$
    Conversely, if such a map $\psi_k:Y\longrightarrow \mathcal{G}r_k(\mathcal{E})$ exists, then there is a map $f : Y\longrightarrow X$ with $\mathcal{F} = \psi_k^*\mathcal{Q}_{\mathcal{E},k}$ a locally free Higgs quotient of $f^*\mathcal{E}$  of rank $k$. 
\vspace{3mm}

For a morphism of varieties $f:Y\longrightarrow X,$ the morphism $\psi_k: Y \longrightarrow \text{Gr}_k(E)$ factors through the Higgs Grassmann scheme
$\mathcal{G}r_k(\mathcal{E})$ if and only if $\theta$ induces a Higgs field on  the universal quotient bundle $\text{Q}_{E,k}$ on $\text{Gr}_{k}(E)$.
\vspace{2mm}

For each $1\leq k \leq r-1,$ we denote the line bundle $\mathcal{O}_{\text{Gr}_k(E)}(1)\vert_{\mathcal{G}r_k(\mathcal{E})}$ by $\mathcal{O}_{\mathcal{G}r(\mathcal{E})}(1).$  For each $k$, we will simply write $\mathcal{Q}_k$ instead of $\mathcal{Q}_{\mathcal{E},k}$ when the Higgs bundle $\mathcal{E}$ is clear from the context.

\subsection{Semi-stability of Higgs vector bundles}  Let $X$ be a a polarized smooth projective variety of dimension $n$ together with a fixed ample line bundle $L$ on it. For a torsion-free Higgs sheaf $\mathcal{E} = (E,\theta)$ on $X$, its slope $\mu_L(\mathcal{E})$ with respect to $L$ is defined as $$\mu_L(\mathcal{E}) := \mu_L(E) = \dfrac{c_1(E)\cdot H^{n-1}}{\rank(E)}.$$
\begin{defi}\label{defn1}
 \rm A  torsion-free Higgs sheaf  $\mathcal{E} = (E,\theta)$ on a smooth projective variety $X$ is said to be \textit{slope Higgs semistable} (respectively, \textit{slope Higgs stable}) with respect to a fixed ample line bundle $L$ on $X$ if  for every $\theta$-invariant proper subsheaf  $\mathcal{G}$ of $\mathcal{E}$ (i.e., $\theta(\mathcal{G}) \subset \mathcal{G}\otimes \Omega_X^1)$, one has
 \begin{center}
  $\mu_L(\mathcal{G})\leq \mu_L(\mathcal{E})$ (respectively, $\mu(\mathcal{G}) < \mu(\mathcal{E}))$.
 \end{center}
A torsion-free Higgs sheaf is called \textit{Higgs unstable} if it is not Higgs semistable.
\end{defi}  
  In fact,  in Definition \ref{defn1} of semistability (respectively, stability) for a Higgs sheaf $\mathcal{E}=(E,\theta)$, it is enough to consider $\theta$-invariant proper subsheaves $\mathcal{G}$ of $\mathcal{E}$  for which the quotient $\mathcal{E}/\mathcal{G}$ is a torsion-free Higgs sheaf.

  When the Higgs field is zero in Definition \ref{defn1}, we recover the usual semistability (respectively, stability) of an ordinary torsion-free sheaf. 
  
  For a  Higgs bundle $\mathcal{E} = (E,\theta)$ on a smooth polarized projective variety $X$ with fixed polarization $L$, there is a unique filtration of  $\theta$-invariant subbundles of $\mathcal{E}$
 \begin{align*}
 0=\mathcal{E}_d\subsetneq \mathcal{E}_{d-1}\subsetneq \cdots \subsetneq \mathcal{E}_1\subsetneq \mathcal{E}_0 = \mathcal{E}
 \end{align*}
such that each Higgs quotient sheaf $\gr^{\HN}_i(\mathcal{E}) = \bigl(\mathcal{E}_i/\mathcal{E}_{i+1},\overline{\theta}\vert_{\mathcal{E}_
{i}}\bigr)$  is a slope Higgs semistable Higgs sheaf with respect to $L$, and 
$$\mu_L(\mathcal{E}_{d-1})> \mu_L(\mathcal{E}_{d-2}/\mathcal{E}_{d-1})>\cdots \cdots> \mu_L(\mathcal{E}_0/\mathcal{E}_1).$$
 Such a filtration is called the\textit{ Higgs Harder-Narasimhan filtration} of $\mathcal{E}$ with respect to $L$.
We define $$\mu^H_{\max}(\mathcal{E}) := \mu_L(\mathcal{E}_{d-1})\hspace{3mm} \text{and} \hspace{2mm} \mu^H_{\min}(\mathcal{E}):= \mu_L(\mathcal{E}_0/\mathcal{E}_1).$$
Note that $$\mu^H_{\max}(\mathcal{E}) \geq \mu_L(\mathcal{E}) \geq \mu^H_{\min}(\mathcal{E})$$ and equality holds if and only if $\mathcal{E}$ is slope semistable with respect to $L$. Note that 
\begin{center}
$\mu^H_{\min}(\mathcal{E}) = \min\bigl\{\mu_L(\mathcal{Q})\mid \mathcal{E} \longrightarrow \mathcal{Q} \longrightarrow 0,\, \mathcal{Q} \neq 0, \,\ \text{$\mathcal{Q}$ is a Higgs quotient}\bigr\}.$
\end{center}
\begin{lem}\label{lem2.2}
 Let $$0\,\longrightarrow\, \mathcal{E}_1\,\longrightarrow\, \mathcal{E} \,\longrightarrow\, \mathcal{E}_2 \,\longrightarrow\, 0$$ be a short exact sequence of Higgs vector bundles over a smooth projective curve $X$, where $\mathcal{E}_1\,=\,(E_1,\,\theta_1)$ is an invariant Higgs subbundle of $\mathcal{E} = (E,\theta)$ and 
    $\mathcal{E}_2\,=\,(E_2,\,\theta_2)$ is the corresponding Higgs quotient. Then $$\mu_{\min}^H(\mathcal{E}_2) \geq \mu^H_{\min}(\mathcal{E}) \geq \min\Bigl\{ \mu_{\min}^H(\mathcal{E}_1),\ \mu_{\min}^H(\mathcal{E}_2)\Bigr\}.$$
    \begin{proof}

    Let $\mathcal{Q}'$ be the unique Higgs quotient of $\mathcal{E}_2$ having minimum slope among all Higgs quotients of $\mathcal{E}_2.$ Now every Higgs quotient of $\mathcal{E}_2$ is also a Higgs quotient of $\mathcal{E}.$ Thus by the definition of $\mu^H_{\min}$, we conclude
    $$\mu^H_{\min}(\mathcal{E}_2) = \mu(\mathcal{Q}') \geq \mu^H_{\min}(\mathcal{E}).$$
         
         Let $\mathcal{E} \longrightarrow \mathcal{Q} \longrightarrow 0$  be the Higgs semistable Higgs quotient of minimal slope in the Higgs Harder–Narasimhan filtration of the Higgs bundle $\mathcal{E}.$ 
         
         Now consider the map $\mathcal{E}_1 \longrightarrow \mathcal{E} \longrightarrow \mathcal{Q}.$ If this induced map  $\mathcal{E}_1 \longrightarrow \mathcal{E} \longrightarrow \mathcal{Q}$ is non-zero, then  its image has slope at most $\mu(\mathcal{Q}) = \mu^H_{\min}(\mathcal{E}).$  If the induced map  $\mathcal{E}_1 \longrightarrow \mathcal{E} \longrightarrow \mathcal{Q}$ is zero, then we have a non-zero map $\mathcal{E}_2 \longrightarrow \mathcal{Q}.$ Again, using the Higgs semistability of $\mathcal{Q}$, we have $\mu_{\min}^H(\mathcal{E}_2)  \leq \mu(\mathcal{Q}) = \mu^H_{\min}(\mathcal{E}).$

         Therefore $$\mu^H_{\min}(\mathcal{E}) \geq \min\Bigl\{ \mu_{\min}^H(\mathcal{E}_1),\ \mu_{\min}^H(\mathcal{E}_2)\Bigr\}.$$
    \end{proof}
\end{lem}
\begin{lem}\label{lem3.3}
    If $\mathcal{E}$ and $\mathcal{F}$ are two Higgs bundles over a smooth curve $C$, then the following hold:
    \begin{enumerate}
        \item $\mu^H_{\min}(\mathcal{E}\otimes \mathcal{F}) = \mu_{\min}^H(\mathcal{E})+ \mu_{\min}^H(\mathcal{F}).$
        \item $\mu_{\min}^H(\mathcal{E}^{\otimes k})= k\, \mu_{\min}^H(\mathcal{E})$ for any non-negative integer $k$.
        \item $\mu^H_{\min}(\Sym^k\mathcal{E}) = k\,  \mu_{\min}^H(\mathcal{E})$ for any non-negative integer $k$.
    \end{enumerate}
    \begin{proof} 
    \begin{enumerate}
        \item We consider the following three cases:
    
     \bf Case (a): \rm If both the Higgs bundles $\mathcal{E}$ and $\mathcal{F}$ are Higgs semistable, then by \cite[Corollary 3.8]{S92} their tensor product is also Higgs semistable. Hence in this case, we get $$\mu^H_{\min}(\mathcal{E}\otimes \mathcal{F}) = \mu(\mathcal{E}\otimes \mathcal{F}) = \mu(\mathcal{E})+\mu(\mathcal{F}) = \mu_{\min}^H(\mathcal{E})+ \mu_{\min}^H(\mathcal{F}).$$ 

     \bf Case (b): \rm Suppose that $\mathcal{E}$ is semistable and  $\mathcal{F}$ is unstable, and let
     \begin{align*}
 0=\mathcal{F}_h\subsetneq \mathcal{F}_{h-1}\subsetneq \mathcal{F}_{h-2}\subsetneq\cdots\subsetneq \mathcal{F}_{1} \subsetneq \mathcal{F}_0 = \mathcal{F}
\end{align*}
 be the Higgs Harder-Narasimhan filtrations for $\mathcal{F}.$ Then, since $\mathcal{E}$ is semistable, the Harder-Narasimhan filtration of $\mathcal{E}\otimes \mathcal{F}$ is given as follows:
  \begin{align*}
 0=\mathcal{E}\otimes \mathcal{F}_h\subsetneq \mathcal{E}\otimes\mathcal{F}_{h-1}\subsetneq \mathcal{E}\otimes \mathcal{F}_{h-2}\subsetneq\cdots\subsetneq \mathcal{E}\otimes \mathcal{F}_{1} \subsetneq \mathcal{E}\otimes \mathcal{F}_0 = \mathcal{E}\otimes \mathcal{F}. 
\end{align*}
Thus $\mu^H_{\min}(\mathcal{E}\otimes \mathcal{F}) = \mu^H_{\min}\bigl((\mathcal{E}\otimes \mathcal{F}_0)/(\mathcal{E}\otimes \mathcal{F}_1)\bigr) = \mu^H_{\min}\bigl(\mathcal{E}\otimes (\mathcal{F}_0/\mathcal{F}_1)\bigr).$
Now $\mathcal{F}_0/\mathcal{F}_1$ is semistable, and hence by Case (a), we conclude:
$$\mu^H_{\min}(\mathcal{E}\otimes \mathcal{F}) = \mu^H_{\min}\bigl(\mathcal{E}\otimes (\mathcal{F}_0/\mathcal{F}_1)\bigr) = \mu^H_{\min}(\mathcal{E}) + \mu_{\min}^H(\mathcal{F}_0/\mathcal{F}_1) = \mu^H_{\min}(\mathcal{E}) + \mu^H_{\min}(\mathcal{F}).$$

\bf Case (c): \rm Suppose that both $\mathcal{E}$ and $\mathcal{F}$ are unstable, and let
\begin{align*}
 0=\mathcal{E}_d\subsetneq \mathcal{E}_{d-1}\subsetneq \cdots \subsetneq \mathcal{E}_1\subsetneq \mathcal{E}_0 = \mathcal{E}
\end{align*}
be the Harder-Narasimhan filtration of $\mathcal{E}.$
We consider the following short exact sequence:
\begin{align}\label{seq2}
0\longrightarrow \mathcal{E}_{d-1}\otimes \mathcal{F} \longrightarrow \mathcal{E}_{d-2}\otimes \mathcal{F} \longrightarrow (\mathcal{E}_{d-2}/\mathcal{E}_{d-1})\otimes \mathcal{F} \longrightarrow 0.
\end{align}
The Higgs bundle $\mathcal{E}_{d-1}$ is Higgs semistable, and thus we have by Case (b)
$$\mu^H_{\min}(\mathcal{E}_{d-1}\otimes \mathcal{F}) = \mu^H_{\min}(\mathcal{E}_{d-1}) + \mu^H_{\min}(\mathcal{F}).$$
As $\mathcal{E}_{d-2}/\mathcal{E}_{d-1}$ is semistable, we also have by Case (b) that 
$$\mu_{\min}^H\bigl((\mathcal{E}_{d-2}/\mathcal{E}_{d-1})\otimes \mathcal{F}\bigr) = \mu_{\min}^H(\mathcal{E}_{d-2}/\mathcal{E}_{d-1}) + \mu^H_{\min}(\mathcal{F}).$$
By Lemma \ref{lem2.2} applied to the short exact sequence (\ref{seq2}), we also have 
$$\mu_{\min}^H\bigl((\mathcal{E}_{d-2}/\mathcal{E}_{d-1})\otimes \mathcal{F}\bigr) \geq \mu^H_{\min}(\mathcal{E}_{d-2} \otimes \mathcal{F})\geq \min\Bigl\{\mu^H_{\min}((\mathcal{E}_{d-2}/\mathcal{E}_{d-1})\otimes \mathcal{F}),\, \mu^H_{\min}(\mathcal{E}_{d-1}\otimes \mathcal{F}) \Bigr\}.$$
By the property of Harder-Narasimhan filtration, we have $$\mu^H_{\min}(\mathcal{E}_{d-1}) = \mu(\mathcal{E}_{d-1}) > \, \mu^H_{\min}\bigl(\mathcal{E}_{d-1}/\mathcal{E}_{d-2} \bigr) = \mu\bigl(\mathcal{E}_{d-1}/\mathcal{E}_{d-2} \bigr),$$
and therefore we conclude that $$\mu^H_{\min}\bigl(\mathcal{E}_{d-2}\otimes \mathcal{F}  \bigr) = \mu^H_{\min}\bigl(\mathcal{E}_{d-1}/\mathcal{E}_{d-2} \bigr) + \mu_{\min}^H(\mathcal{F}).$$
Now proceeding inductively, we get
$$\mu^H_{\min}\bigl(\mathcal{E}\otimes \mathcal{F} \bigr) = \mu^H_{\min}\bigl(\mathcal{E}_{0}/\mathcal{E}_{1} \bigr) + \mu^H_{\min}(\mathcal{F}) = \mu_{\min}^H(\mathcal{E})+\mu_{\min}^H(\mathcal{F}).$$
This completes the proof.
\item Repeated application of (1) gives the proof.
\item  Let $\mathcal{Q}$ be the Higgs quotient of $\Sym^k\mathcal{E}$ in the Harder-Narasimhan filtration of $\Sym^k\mathcal{E}$ having minimum slope among all Higgs quotients of $\Sym^k\mathcal{E}.$ Note that $\Sym^k\mathcal{E}$ is a non-zero Higgs quotient of $\mathcal{E}^{\otimes k},$ and hence every non-zero Higgs quotient of $\Sym^k\mathcal{E}$ is also a Higgs quotient of $\mathcal{E}^{\otimes k}.$ Thus by the definition of $\mu^H_{\min}$ and by part (2) we have
$$ k\, \mu^H_{\min}(\mathcal{E}) = \mu_{\min}^H (\mathcal{E}^{\otimes k})\leq \mu(\mathcal{Q})= \mu_{\min}^H(\Sym^k\mathcal{E}).$$
To prove the reverse inequality, we consider the Higgs quotient $\mathcal{Q'}$ of $\mathcal{E}$ having the minimum slope. Then $\Sym^k\mathcal{Q}'$ is a non-zero Higgs quotient of $\Sym^k\mathcal{E}.$ Therefore by the definition of $\mu^H_{\min}$ we have
$$\mu^H_{\min}(\Sym^k\mathcal{E}) \leq \mu(\Sym^k\mathcal{Q}') = k \, \mu(\mathcal{Q}') = k\, \mu^H_{\min}(\mathcal{E}).$$
This completes the proof.
\end{enumerate}
    \end{proof}
\end{lem}
\begin{defi}\label{2.4}
\rm A Higgs bundle $\mathcal{E}$ on a 
smooth projective variety $X$ 
is \textit{curve Higgs semistable} if for every morphism $f:C\longrightarrow X$ where $C$ is a smooth irreducible projective curve, the pullback Higgs bundle $f^*\mathcal{E}$ is Higgs semistable. 
\end{defi}

When the Higgs field is zero in Definition \ref{2.4}, we get back the usual curve semistability of ordinary vector bundles. Note that if a vector bundle $E$ is curve semistable, then the Higgs bundle $\mathcal{E} = (E,\theta)$ is curve Higgs semistable for any Higgs field $\theta$.

The discriminant of a Higgs bundle $\mathcal{E}=(E,\theta)$ of rank $r$, denoted by $\Delta(E)$, is defined as follows:
$$\Delta(E) = 2rc_2(E)-(r-1)c_1^2(E).$$
We have the following conjecture due to Ugo Bruzzo and Beatriz Gra\~{n}a Otero (see \cite[Theorem Theorems 4.7 and A.5]{BO11}).
\begin{conj}\cite[Conjecture 1.2]{BOR23}
    Let $\mathcal{E}=(E,\theta)$ be a Higgs bundle on a smooth projective variety $X$ of dimension $n$. Then the following are equivalent:
    \begin{enumerate}
        \item $\mathcal{E}$ is semistable with respect to some polarization $H$ and $\Delta(E)\cdot H^{n-2}=0$.
        \item For any morphism $f:C \longrightarrow X$, where $C$ is a smooth projective curve, the Higgs bundle $f^*\mathcal{E}$ is Higgs semistable, i.e., the Higgs bundle $\mathcal{E}$ is curve Higgs semistable.
    \end{enumerate}
\end{conj}
The fact that condition (1) implies condition (2) in the above conjecture was proved in \cite{BR06}. See \cite{BG16}, \cite{BLG19} and \cite{BOR23} for more details about this conjecture.

\section{Higgs nef and Higgs ample bundles}\label{Higgs-ample}

An ordinary vector bundle $E$ on a projective scheme $X$  is called \textit{ample} (respectively, \textit{nef}) if the tautological line bundle $\mathcal{O}_{\mathbb{P}(E)}(1)$ on its projectivization $\mathbb{P}_X(E)$ is ample (respectively, nef).
We now recall the notions of nefness and 
ampleness of Higgs bundles. 
\begin{defi}\label{defi-nef}

\rm \cite[Definition 2.3]{BBG19} Let $\mathcal{E} = (E,\theta)$ be a Higgs vector bundle of rank $r$. If $\rank(E) = 1$, then we say $\mathcal{E}$ is \textit{Higgs nef} ($\text{H}$-nef in short) if and only if $E$ is nef as a line bundle in the usual sense. If $\rank(E) \geq 2$, then the Higgs nefness ($\text{H}$-nefness in short) of $\mathcal{E}$ is defined inductively by requiring that:
\begin{enumerate}
    \item for all $1\leq k \leq r-1$, the universal Higgs quotient bundles $\mathcal{Q}_{\mathcal{E},k}$ on Higgs Grassmann schemes $\mathcal{G}r_k(\mathcal{E})$ are Higgs nef, and 
    \item the line bundle $\det(E)$ is nef in the usual sense.
\end{enumerate}
\end{defi}
\begin{xrem}\label{rem-1}
    Let $E$ be a nef vector bundle of rank $r$ on a projective variety $X$. Then the Higgs bundle $\mathcal{E} = (E,\theta)$ is Higgs nef for any Higgs field $\theta$. We prove this statement by induction on the rank of $\mathcal{E}.$ If $\rank(\mathcal{E})$ is one, then $\mathcal{E}$ is Higgs nef if and only if $E$ is nef. Now assume that $\rank(\mathcal{E}) \geq 2.$ Recall the  exact sequence 
    \eqref{seq-Grassmann}
    for each $1\leq k\leq r-1$:
    $$0\longrightarrow S_{E,r-k} \longrightarrow p_k^*E \longrightarrow Q_{E,k} \longrightarrow 0.$$ The vector bundle $p_k^*E$ is nef since it is the pullback of a nef vector bundle. Since quotient of a nef vector bundle is nef, we have $Q_{E,k}$ is nef. As $\mathcal{Q}_{E,k} = Q_{E,k}\vert_{\mathcal{G}r_k(\mathcal{E})}$, we conclude that the underlying vector bundle of the Higgs vector bundle $\mathcal{Q}_{E,k}$ is nef. Thus by induction hypothesis, we have that $\mathcal{Q}_{E,k}$ is Higgs nef for every $1\leq k \leq r-1.$ Also, $\det(E)$ is nef as $E$ is nef. Therefore $\mathcal{E}$ is Higgs nef.
\end{xrem}
\begin{xrem}\label{rem-nef}\rm
  For a Higgs bundle $\mathcal{E} =(E,\theta)$ of rank $r$ with the Higgs field $\theta = 0,$ the Higgs nefness of $\mathcal{E} = (E,\theta)$ is equivalent to the usual nefness of the underlying vector bundle $E$. In this case, we have $\mathcal{G}r_k(\mathcal{E}) = \Gr_{k}(E)$ and $\mathcal{O}_{\mathcal{G}r_k(\mathcal{E})}(1) = \mathcal{O}_{\Gr_k(E)}(1)$ for each $1\leq k \leq r-1$. Now Higgs nefness of $\mathcal{E}$ implies that $\det(\mathcal{Q}_{\mathcal{E},k}) = \mathcal{O}_{\mathcal{G}r_k(\mathcal{E})}(1)$ is nef for every $1\leq k \leq r-1.$ In particular, $\mathcal{O}_{\mathbb{P}(E)}(1) = \mathcal{O}_{\Gr_1(E)}(1) = \mathcal{O}_{\mathcal{G}r_1(\mathcal{E})}(1)$ is nef, and thus $E$ is nef in the usual sense.
\end{xrem}

\begin{prop}\label{prop4.5}
    Let $\mathcal{E} = (E,\theta)$ be a Higgs bundle of rank $r$ on a smooth projective variety $X$. Then the following are equivalent:
    \begin{enumerate}
    \item The Higgs bundle $\mathcal{E}$ is Higgs nef.
     \item The line bundles $\det(E)$ and $\mathcal{O}_{\mathcal{G}r_k(\mathcal{E})}(1)$ are nef for all $1\leq k\leq r-1.$
        \item For any smooth curve $C$ and non-constant map $f:C\longrightarrow X,$ we have $\mu^H_{\min}(f^*\mathcal{E}) \geq 0.$
       
         \end{enumerate}
\end{prop}
\begin{proof}
    (1) $\implies$ (2): Suppose that $\mathcal{E}$ is Higgs nef on $X$. Then by definition, $\det(E)$ is nef and for all $1\leq k \leq r-1$, the universal Higgs quotient bundles $\mathcal{Q}_{\mathcal{E},k}$ on Higgs Grassmann bundles $\mathcal{G}r_k(\mathcal{E})$ are Higgs nef. Therefore, $\mathcal{O}_{\mathcal{G}r_k(\mathcal{E})}(1) = \det(\mathcal{Q}_{\mathcal{E},k})$ are nef for all $1\leq k\leq r-1$.
\vspace{2mm}

    (2) $\implies$ (3): Suppose $\det(E)$ and $\mathcal{O}_{\mathcal{G}r_k(\mathcal{E})}(1)$ are nef for all $1\leq k \leq r-1$.  We just need to prove the following:
    for any non-constant map $f\,:\,C\,\longrightarrow\, X$ from a smooth projective curve $C$ to $X$,
    $$\mu_{\min}^{H}(f^*\mathcal{E})\ \geq\  0.$$

    We consider the following two cases: 

    \begin{enumerate}
        \item {\bf Case 1}:\  Suppose that $f^*\mathcal{E}$ is a slope semistable Higgs vector bundle.
        Then $$\mu_{\min}^H\bigl(f^*\mathcal{E}\bigr)\, = \,\mu\bigl(f^*\mathcal{E}\bigr)\, =\, \dfrac{\det(E)\cdot\widetilde{C}}{r} \,\geq\, 0,$$
        where $\widetilde{C}$ is the image of $C$ in $X$.

         \item {\bf Case 2}:\ Assume that $f^*\mathcal{E}$ is not Higgs semistable, and let \begin{align*}
  0 \,= \,\mathcal{E}_d \,\subsetneq\, \mathcal{E}_{d-1} \,\subsetneq\, \mathcal{E}_{d-2} \,\subsetneq\, \cdots\,\subsetneq \,\mathcal{E}_1 \,\subsetneq\, \mathcal{E}_0\,=\, f^*\mathcal{E}
 \end{align*}
 be the Harder-Narasimhan filtration of $f^*\mathcal{E}.$ Let $s\,=\, \rank\bigl( f^*\mathcal{E}/\mathcal{E}_1\bigr)$. Then by the universal property of Higgs Grassmannian, there exists a lift 
 $f_s \,:\, C \,\longrightarrow \,\mathcal{G}r_s(\mathcal{E})$ of $f$ such that 
 $$\bigl(f^*\mathcal{E}/\mathcal{E}_1\bigr)\ =\ f_s^*\mathcal{Q}_{\mathcal{E},s}.$$
 
 As $\mathcal{O}_{\mathcal{G}r_s(\mathcal{E})}(1) \,=\, \det(\mathcal{Q}_{\mathcal{E},s})$ is nef, and since the pullback of a nef line bundle is nef, we have
 \begin{align*}
     \mu_{\min}(f^*\mathcal{E}) \,=\, \mu(f^*\mathcal{E}/\mathcal{E}_1) = \dfrac{1}{s}\deg(f_s^*\mathcal{Q}_{\mathcal{E},s})\, =\, \dfrac{f_s^*\bigl(\det(\mathcal{Q}_{\mathcal{E},s})\bigr)\cdot C}{s} \,\geq\,0.
 \end{align*}
  
 \end{enumerate}

(3) $\implies$ (1): The equivalence of (1) and (3) is given by \cite[Lemma 3.3]{BBG19}.
 
\end{proof}

\begin{xrem}
    In view of the above Proposition \ref{prop4.5}, we can define alternatively that a Higgs bundle $\mathcal{E} = (E,\theta)$ of rank $r$ is Higgs nef if and only if its determinant bundle $\det(E)$ is a nef line bundle and $\mathcal{O}_{\mathcal{G}r_k(\mathcal{E})}(1)$ are nef for all $1\leq k\leq r-1.$
\end{xrem}
\begin{defi}\label{defi2}
\rm  \cite[Definition 2.5]{BMR26}  Let $\mathcal{E} = (E,\theta)$ be a Higgs vector bundle of rank $r$ on a  projective variety $X$. We say $\mathcal{E}$ is \textit{Higgs ample} ($\text{H}$-ample in short) if it is H-nef and the following two conditions are satisfied:
    \begin{enumerate}
        \item for all $1\leq k \leq r-1$, the line bundles $\mathcal{O}_{\mathcal{G}r_k(\mathcal{E})} (1)$  on Higgs Grassmann schemes $\mathcal{G}r_k(\mathcal{E})$ are ample in the usual sense, and 
        \item $\det(E)$ is ample as a line bundle in the usual sense.
    \end{enumerate}
\end{defi}

\begin{xrem}\label{3.7}
    \rm If $E$ is an ample vector bundle, then the Higgs vector bundle $\mathcal{E}=(E,\theta)$ is Higgs ample for any Higgs field $\theta$ according to Definition \ref{defi2}.  Indeed, if $E$ is ample, then for any $1\leq k \leq \rank(E)$, its exterior powers $\wedge^kE$ are also ample, so by the Pl\"{u}cker embedding $\omega : \text{Gr}_k(E) \hookrightarrow \mathbb{P}(\wedge^kE)$, we see that  $\mathcal{O}_{\text{Gr}_k(E)}(1) = \omega^*\mathcal{O}_{\mathbb{P}(\wedge^k(E))}(1)$ is ample. Thus for any Higgs field $\theta$,  $\mathcal{O}_{\mathcal{G}r_k(\mathcal{E})}(1) = \mathcal{O}_{\text{Gr}_k(E)}(1)\vert_{\mathcal{G}r_k(\mathcal{E})}$  are ample. On the other hand, $\det(E)$ is also ample as $E$ is ample. This implies $\mathcal{E}=(E,\theta)$ is Higgs ample.
    \end{xrem}
 
\begin{xrem}
 \rm  According to Definition \ref{defi2}, when the Higgs field $\theta$ is 0, the Higgs bundle $\mathcal{E} = (E,0)$ is Higgs ample if and only if $E$ is ample as a vector bundle in the usual sense. This is indeed true as $\mathcal{G}r_k(E) = \text{Gr}_k(E)$ for every $k$ when the Higgs field $\theta$ is zero. If 
 $\mathcal{E}$ is Higgs ample then 
$\mathcal{O}_{\Gr_k(E)}(1)$ are ample for every $k$ implying that $E$ is ample. The converse follows from Remark \ref{3.7}. Therefore,  Definition \ref{defi2} is the natural generalization of the usual ampleness of ordinary vector bundles. An alternative definition of Higgs ampleness is also considered in \cite{BCO25}.
\end{xrem}

\section{Seshadri constants of Higgs bundles}\label{sc}

In this section, we define Seshadri constants for Higgs bundles and prove some of their fundamental properties. 
\begin{defi}\rm
    Let $\mathcal{E} = (E,\, \theta)$ be an \rm $\text{H}$\it-nef Higgs bundle of rank $r$ on a smooth projective variety $X$. Then the Higgs Seshadri constant of $\mathcal{E}$ at a point $x\in X$ is defined as follows:
    $$ \varepsilon^H(\mathcal{E}, x) := \min\Bigl\{\min_{1\,\leq k\, \leq r-1} \varepsilon(\xi_k;x) , \,\ \dfrac{1}{r}\, \varepsilon\bigl(\det(E);x\bigr)\Bigr\},$$
    where $\xi_k$ denotes the numerical class of $\mathcal{O}_{\mathcal{G}r_k(\mathcal{E})}(1)$ on the Higgs Grassmann $\mathcal{G}r_k(\mathcal{E}).$
\end{defi}

\begin{xrem}\rm
   If $\mathcal{E} = (E,\, \theta)$ is a \rm $\text{H}$\it-nef Higgs bundle of rank $1$ on a smooth projective variety $X$, then the Higgs Seshadri constant $\varepsilon^H(\mathcal{E};x)$  coincides with the usual Seshadri constant $\varepsilon(E;x)$ of the underlying line bundle $E$ at every point $x\in X$.
\end{xrem}

\begin{xrem}
\rm 
 For a $\text{H}$-nef Higgs bundle $\mathcal{E}=(E,\theta)$ of rank $r$ on a smooth projective variety $X$, recall the inclusions $i: \mathcal{G}r_k(\mathcal{E}) \hookrightarrow \text{Gr}_k(E)$ and the Pl\"{u}cker embedding $\omega : \text{Gr}_k(E) \hookrightarrow \mathbb{P}(\wedge^kE)$ such that $\xi_k = i^*\eta_k$ and $\eta_k = \omega^*\phi_k$, where $$\xi_k \equiv \mathcal{O}_{\mathcal{G}r_k(\mathcal{E})}(1) ,\, \eta_k \equiv \mathcal{O}_{\text{Gr}_k(E)}(1),\  \phi_k = \mathcal{O}_{\mathbb{P}(\wedge^k(E))}(1).$$
    Then by \cite[Lemma 3.18]{FM21} for every $1\leq k \leq r-1$ and for every point $x\in X$, we have for the usual Seshadri constants
    $$ \varepsilon(\xi_k\,;x) \geq \varepsilon(\eta_k\,;x) \geq \varepsilon(\phi_k\,;x).$$
    Note that using \cite[Corollary 3.21]{FM21}, we have $$\varepsilon(\phi_k\,;x) = \varepsilon(\wedge^kE;x) = \inf_{x\in C \subseteq X}\Bigl\{ \dfrac{\mu_{\min}(\nu^*(\wedge^k(E)))}{\mult_xC}\Bigr\}.$$
    Now note that since $\nu^*(\wedge^kE)$ is a quotient of $\nu^*(E^{\otimes k})$, every quotient of $\nu^*(\wedge^kE)$ is also a quotient of $\nu^*(E^{\otimes k})$. Thus we have by the definition of $\mu_{\min}$ that $$\mu_{\min}(\nu^*(\wedge^kE)) \geq \mu_{\min}(\nu^*(E^{\otimes k})) = k\, \mu_{\min}(\nu^*E).$$
    Since $\mu_{\min}(\nu^*\bigl(\wedge^k(E)\bigr) \geq  k\, \mu_{\min}(\nu^*E)$,  we conclude that 
    \begin{align}\label{s0}
    \varepsilon(\xi_k\,;x) \geq \varepsilon(\eta_k\,;x) \geq \varepsilon(\phi_k\,;x) \geq  k\,\inf_{x\in C \subseteq X}\Bigl\{ \dfrac{\mu_{\min}(\nu^*(E))}{\mult_xC}\Bigr\} = k\, \varepsilon(E;x) = k \, \varepsilon(\eta_1;x).
    \end{align}
 \end{xrem}
 
 \begin{xrem}
  \rm When the Higgs field $\theta =0$, then the $\text{H}$-nefness of a Higgs bundle $\mathcal{E}=(E,\theta)$ is equivalent to the usual nefness of the underlying vector bundle $E$ (see Remark \ref{rem-nef}). In such case, for every $1\leq k \leq r-1$, we have $\mathcal{G}r_k(\mathcal{E}) = \Gr_k(E)$ and $\xi_k=\eta_k$. Also, the nefness of $E$ implies $\varepsilon(E;x) = \varepsilon(\eta_1;x) \geq 0$ for all $x\in X$. Now from (\ref{s0}) we conclude that for all  $1\leq k \leq r-1$ $$ \varepsilon(\xi_k;x) = \varepsilon(\eta_k;x) \geq k\, \varepsilon(E;x) = k \, \varepsilon(\eta_1;x) \geq \varepsilon(\eta_1;x). $$
 Thus $$\min\limits_{1\,\leq k\, \leq r-1} \varepsilon(\xi_k; x) = \min\limits_{1\,\leq k\, \leq r-1} \varepsilon(\eta _k; x) = \varepsilon(\eta_1;x) = \varepsilon(E;x).$$
 Also by \cite[Lemma 3.28]{FM21}, $$\varepsilon(E;x) \leq \dfrac{1}{r} \varepsilon(\det(E);x)$$ for every point $x\in X$.
 
 Thus $$ \varepsilon^H(\mathcal{E}; x) = \min\Bigl\{ \min_{1\,\leq k\, \leq r-1} \varepsilon(\xi_k; x), \,\ \dfrac{1}{r} \varepsilon\bigl(\det(E);x\bigr)\Bigr\} = \varepsilon(E;x),$$ i.e., in the zero Higgs field case, the Higgs Seshadri constants of a Higgs nef Higgs bundle coincide with the usual Seshadri constants for nef vector bundles.
\end{xrem}
\begin{xrem}
\rm 
Let $E$ be a nef vector bundle on a smooth projective variety $X$. Then $\mathcal{E} = (E,\theta)$ is a nef Higgs bundle for any Higgs field $\theta$ (see Remark \ref{rem-1}). In this case, by taking minimum on both sides of \rm (\ref{s0}) over the set $1\leq k \leq r-1$, we get 
    \begin{align*}
   \min_{1\,\leq k\, \leq r-1} \varepsilon(\xi_k;x) \geq \min_{1\,\leq k\, \leq r-1} k\,\ \varepsilon(\eta_1;x) = \varepsilon(\eta_1;x) = \varepsilon(E;x)
    \end{align*}
    for every point $x\in X.$  Also by \cite[Lemma 3.28]{FM21} $$\varepsilon(E;x) \leq \dfrac{1}{r} \varepsilon(\det(E);x)$$ for every point $x\in X$.
    Thus we conclude that 
    \begin{align}\label{s1}
    \varepsilon^H(\mathcal{E}; x) = \min\Bigl\{ \min_{1\,\leq k\, \leq r-1} \varepsilon(\xi_k; x), \,\ \dfrac{1}{r} \varepsilon\bigl(\det(E);x\bigr)\Bigr\} \geq \varepsilon(E;x).
    \end{align}
\end{xrem}

\begin{exm}
\rm The inequality in  (\rm \ref{s1}) \rm can be strict as  Example \ref{exm3.4} below shows.
\end{exm}

\begin{thm}\label{thm-curves}
     Let $\mathcal{E} = (E,\, \theta)$ be a Higgs nef Higgs bundle of rank $r$ on a smooth projective curve $X$. Then 
     $$ \varepsilon^H(\mathcal{E}; x) = \mu^H_{\min}(\mathcal{E})$$ for all points $x\in X$.
    
     \begin{proof}
     
      Note that by \cite[Remark 3.22]{FM21} for any point $x\in X$ \begin{center}
    $ \varepsilon(\xi_k;\, x) = \sup\Bigl\{t\geq 0 \mid \xi_k-tf_k$ is nef$\Bigr\},$ 
    \end{center}
    where $f_k$ denotes the numerical class of a fiber of the map $\rho_k : \mathcal{G}r_k(\mathcal{E}) \longrightarrow X$.
    We consider the following two cases: 
%     \begin{enumerate}
         %\item
        
         \bf Case 1: \rm
    We first prove the theorem under the assumption that $\mathcal{E}$ is Higgs semistable. Since $\mathcal{E}$ is Higgs semistable, using \cite[Theorem 1.2]{BR06}, we have that,
      for each $1\leq k\leq r-1$,\begin{center}
     $\sup\Bigl\{ t \geq 0\mid \xi_k - tf_k $ is nef $\Bigr\} = k\, \mu(\mathcal{E}).$
     \end{center}
     As $\mathcal{E}$ is Higgs semistable and $\textrm{H}$-nef, by \cite[Lemma 3.3]{BBG19}
     $$\mu_{\min}^H(\mathcal{E}) = \mu(\mathcal{E}) \geq 0.$$
     Thus in this case, for any point $x\in X$ 
     \begin{align}\label{seq10}
     \min\limits_{1\leq k \leq r-1}\varepsilon(\xi_k;x) = \min\limits_{1\leq k \leq r-1} k \, \mu_{\min}^H(\mathcal{E}) = \min\limits_{1\leq k \leq r-1} k \, \mu(\mathcal{E}) = \mu(\mathcal{E}) = \mu_{\min}^H(\mathcal{E}).
     \end{align}
     Since $X$ is a smooth curve, we have $\dfrac{1}{r}\, \varepsilon\bigl(\det(E);x) = \dfrac{\deg(E)}{r} = \mu(\mathcal{E}) = \mu^H_{\min}(\mathcal{E})$ for every point $x$ in $X$. Therefore,
      $$\varepsilon^H(\mathcal{E}; x) = \min\Bigl\{ \min_{1\,\leq k\, \leq r-1} \varepsilon(\xi_k; x), \,\ \dfrac{1}{r} \varepsilon\bigl(\det(E);x\bigr)\Bigr\} = \mu_{\min}^H(\mathcal{E}).$$
  %\item 
  
  \bf Case 2: \rm 
     We  assume from now on  that $\mathcal{E}$ is Higgs unstable.
     As $\mathcal{E}$ is Higgs nef,  by \cite[Lemma 3.3]{BBG19}, $\mu^H_{\min}(\mathcal{E}) \geq 0.$ We consider the following two subcases:
     
   %  \begin{enumerate}
         %\item  
         \bf Subcase 2(a): \rm We first prove the theorem under the assumption that $\mu^H_{\min}(\mathcal{E})=0.$ Then there exists a Higgs semistable Higgs quotient bundle $\mathcal{Q}$ of $\mathcal{E}$ having minimal slope such that $\mu^H_{\min}(\mathcal{E})=\mu(\mathcal{Q})=0$.  Thus by (\ref{seq10}) in Case 1, for every point $x\in X$ $$\varepsilon^H(\mathcal{Q};x) = \mu^H_{\min}(\mathcal{Q}) = \mu(\mathcal{Q}) = 0.$$ 

         Our claim is that $\varepsilon^H(\mathcal{E};x) = 0.$ Note that for any $1\leq k \leq \rank(\mathcal{Q})-1,$ we have the following diagram  \begin{center}
   \begin{tikzcd}
\mathcal{G}r_k(\mathcal{Q}) \arrow[rd, "\rho_k'"] \arrow[r, "i"] & \mathcal{G}r_k(\mathcal{E}) \arrow[d,"\rho_k"]\\
& X
\end{tikzcd}
  \end{center}
such that $i^*\xi_{\mathcal{E},k} = \xi_{\mathcal{Q},k}.$ Hence using  \cite[Lemma 3.19]{FM21} we have 
for every point $x\in X$
         $$\varepsilon(\xi_{\mathcal{Q},k};x) = \varepsilon(i^*\xi_{\mathcal{E},k} ;x) \geq \varepsilon(\xi_{\mathcal{E},k};x) \geq \varepsilon^H(\mathcal{E};x).$$
         As $\mathcal{Q}$ is Higgs semistable, we have 
         $$\varepsilon(\xi_{\mathcal{Q},k};x) = \sup\Bigl\{ t\geq 0 \mid \xi_{\mathcal{Q},k} - t F_k \hspace{2mm} \text{is nef}\Bigr\} = k \,\ \mu(\mathcal{Q}) = k\,\ \mu_{\min}^H(\mathcal{E}) = 0,$$
         where $F_k$ denotes the fiber of the map $\rho'_k: \mathcal{G}r_k(\mathcal{Q})\longrightarrow X.$
         Since $\mathcal{E}$ is Higgs nef, we always have $\varepsilon^H(\mathcal{E};x) \geq 0.$ This shows that $\varepsilon^H(\mathcal{E};x) = 0$ and proves our claim.

         Therefore, in this subcase, for every point $x\in X$ $$ \varepsilon^H(\mathcal{E};x) = \mu_{\min}^H(\mathcal{E}) = 0.$$

         %\item 
         \bf Subcase 2(b): \rm Now we assume that $\mathcal{E}$ is Higgs unstable and $\mu^H_{\min}(\mathcal{E}) > 0$. Let $L$ be a line bundle of negative degree on $X$ such that $\mu^H_{\min}(\mathcal{E}\otimes L) = \mu^H_{\min}(\mathcal{E}) + \deg(L) = 0$. 
         
         Recall that 
         \begin{align}\label{s12}
         \mu(\mathcal{E}\otimes L) \geq \mu_{\min}^H(\mathcal{E}\otimes L) = 0.
         \end{align}
         
         So $\mathcal{E}\otimes L$ is Higgs nef by Proposition \ref{prop4.5}.
          We first note that for any $1\leq k\leq r-1$,  $\xi_{{\mathcal{E}\otimes L},k} =\xi_{\mathcal{E},k} \otimes \rho_k^*L^{\otimes k}.$ 
Now for any curve $C\in \mathcal{C}_{\rho_k,x}$, $$\dfrac{\xi_{{\mathcal{E}\otimes L},k}}{\mult_x \rho_k{_*}C} = \dfrac{\xi_{\mathcal{E},k}\cdot C + \rho_k^*L^{\otimes k}\cdot C}{\mult_x \rho_k{_*}C} =  \dfrac{\xi_{\mathcal{E},k}\cdot C }{\mult_x \rho_k{_*}C} + \dfrac{k\,\ \rho_k^*L\cdot C}{\mult_x \rho_k{_*}C}.$$
Using the projection formula,  we get 
\begin{align*}
\dfrac{\xi_{\mathcal{E},k}\cdot C }{\mult_x \rho_k{_*}C} + \dfrac{k\,\ \rho_k^*L\cdot C}{\mult_x \rho_k{_*}C} & = \dfrac{\xi_{\mathcal{E},k}\cdot C }{\mult_x \rho_k{_*}C} + \dfrac{k\,\ L\cdot \rho_k{_*} C}{\mult_x \rho_k{_*}C}\\
&=  \dfrac{\xi_{\mathcal{E},k}\cdot C }{\mult_x \rho_k{_*}C} + k\deg(L).
\end{align*}
Taking infimum over the curves $C\in \mathcal{C}_{\rho_k,x}$, 
%where $\rho_k:\mathcal{G}r_k(\mathcal{E})\longrightarrow X$, 
we get $$\varepsilon(\xi_{{\mathcal{E}\otimes L},k};x) = \varepsilon(\xi_{\mathcal{E},k};x) + k\deg(L)$$ for each $1\leq k \leq r-1$.

As $\mathcal{E}\otimes L$ is Higgs nef with $\mu_{\min}^H(\mathcal{E}\otimes L) = 0$, by Subcase 2(a), we have $$\varepsilon^H(\mathcal{E}\otimes L;x) = \mu_{\min}^H(\mathcal{E}\otimes L) = 0$$ for every point $x\in X$.

Hence
$$0=\varepsilon^H(\mathcal{E}\otimes L;x) \leq \varepsilon(\xi_{{\mathcal{E}\otimes L},k};x) = \varepsilon(\xi_{\mathcal{E},k};x) + k\deg(L).$$
This, in particular, implies that for each $1\leq k \leq r-1,$$$\varepsilon(\xi_{\mathcal{E},k};x) \geq -k\deg(L) \geq -\deg(L)$$
as $\deg(L) <0.$ Also, using (\ref{s12}) and the fact that $X$ is a smooth curve, we have $$\dfrac{1}{r}\,\ \varepsilon(\det(E);x) = \mu(\mathcal{E}) \geq -\deg(L).$$
         Therefore, we conclude that for every point $x\in X$ $$\varepsilon^H(\mathcal{E};x) \geq - \deg(L),\,\ \,\ \text{i.e.}, \,\ \ \varepsilon^H(\mathcal{E};x) + \deg(L) \geq 0.$$

         We also have $\mu_{\min}^H(\mathcal{E}\otimes L) = \mu_{\min}^H(\mathcal{E})+\deg(L) = 0.$
Hence $\varepsilon^H(\mathcal{E};x) \geq \mu^H_{\min}(\mathcal{E}).$

To prove the reverse inequality, just note that $$\varepsilon^H(\mathcal{E};x) \leq \varepsilon(\xi_{\mathcal{E},k};x) \leq \varepsilon(\xi_{\mathcal{Q},k};x) = k\,\ \mu^H(\mathcal{Q}) = k\,\ \mu^H_{\min}(\mathcal{E}),$$ for every $k$ and for every $x\in X,$ where $\mathcal{Q}$ is the unique torsion-free Higgs semistable Higgs quotient of $\mathcal{E}$ having the smallest slope among all the non-zero Higgs quotient. In particular,  as $\mu^H_{\min}(\mathcal{E})\geq 0$, we have for every point $x\in X$, $\varepsilon^H(\mathcal{E};x) \leq \mu_{\min}^H(\mathcal{E})$.
%\end{enumerate}

This completes the proof.
 %    \end{enumerate}
\end{proof}
\end{thm}
\begin{exm}\label{exm3.4}
\rm Let us consider a nilpotent Higgs bundle $\mathcal{E}=(E,\theta)$ on a smooth curve $C$ with $E=L_1\oplus L_2.$ Here $L_1,L_2$ are line bundles with $\deg(L_1) = 1,$ and $\deg(L_2) =0.$ Further assume that 
the Higgs field is given by $\theta: L_1\longrightarrow L_2\otimes \Omega^1_C$ and $\theta(L_2)=0.$

It is shown in \cite{BR06} that $\mathcal{E}$ has only two non-trivial proper Higgs quotients, i.e., $L_1$ and $Q_{tor}$, where $Q= \text{coker} (\theta\otimes id): E \otimes T_C\longrightarrow E$ and $Q_{tor}$ is $Q$ modulo torsion. Note that $\deg(Q_{tor}) \geq \deg(L_1)$. In this case, $$\mu_{\min}^H(\mathcal{E}) = \min\Bigl\{\mu(L_1),\, \mu(Q_{tor}),\, \mu(E)\Bigr\} = \mu(E) = \dfrac{1}{2}, $$
whereas the usual Harder-Narasimhan filtration of the ordinary bundle $E$ is given by: $$0 \subsetneq L_1\subsetneq E$$
so that $\mu_{\min}(E) = \deg(L_2) = 0.$ Note that the vector bundle $E$ is nef on the smooth curve $C$ as it is the direct sum of two nef line bundles, and hence the Higgs bundle $\mathcal{E}$ is Higgs nef by Remark \ref{rem-1}. However, for any point $x\in C$, the Higgs Seshadri constant is bigger than the usual Seshadri constant: 
$$\varepsilon^H(\mathcal{E};x) = \mu_{\min}^H(\mathcal{E}) = \dfrac{1}{2} >  \varepsilon(E;x) = \mu_{\min}(E) = \deg(L_2) = 0.$$
So the inequality in (\rm \ref{s1}) can be strict. 
Note that in this example, the Higgs bundle $\mathcal{E}$ is Higgs semistable, but the underlying vector bundle $E$ is not semistable.
\end{exm}
\begin{thm}\label{ses-criteria}
    (Seshadri's ampleness criterion for Higgs bundles)  A $\text{H}$-nef  Higgs bundle $\mathcal{E}=(E,\theta)$ of rank $r$ on a smooth projective variety $X$ is $\text{H}$-ample if and only if $$\inf_{x\in X}\varepsilon^H(\mathcal{E};x) > 0.$$
    \begin{proof}
        First suppose that $A:=\inf\limits_{x\in X} \varepsilon^H(\mathcal{E};x) > 0$. We will show that $\mathcal{E}$ is Higgs ample. We need to show that $\xi_k \equiv \mathcal{O}_{\mathcal{G}r_k(\mathcal{E})}(1)$ is ample for every $1\leq k \leq r-1,$ and $\det(E)$ is ample. 
        For a fixed $x\in X$ and for a fixed $k$, we have that, by the definition of Seshadri constants, $$\varepsilon(\xi_k;x) \geq \varepsilon^H(\mathcal{E};x)  \geq A >0, \hspace{3mm} \text{and} 
        \hspace{3mm} \dfrac{1}{r}\, \varepsilon(\det(E);x) \geq \varepsilon^H(\mathcal{E};x) \geq A >0. $$
        In particular, for each $k$ with $1\leq k\leq r-1$  $$\inf_{x\in X} \varepsilon(\xi_k;x) \geq A >0, \hspace{2mm} \text{and} \hspace{2mm} \dfrac{1}{r}\, \inf\limits_{x\in X} \varepsilon(\det(E);x) \geq A >0.$$
        Then by Seshadri's ampleness criterion \cite[Theorem 3.11]{FM21}, each $\xi_k$ is ample and $\det(E)$ is ample by \cite[Theorem 1.4.13]{L1}.

        Conversely, let $\mathcal{E}$ be $\text{H}$-ample.  We show that $$\inf_{x\in X}\varepsilon^H(\mathcal{E};x) > 0.$$ Note that as $\mathcal{E}$ is $H$-ample, $\det(E)$ is ample and each $\xi_k$ is ample for every $1\leq k \leq r-1$. So by Seshadri's ampleness criterion \cite[Theorem 3.11]{FM21}, we have for each $1\leq k \leq r-1$
        $$\inf_{x\in X} \varepsilon(\xi_k;x) > 0 \hspace{3mm} \text{and} \hspace{3mm} \dfrac{1}{r}  \inf\limits_{x\in X} \varepsilon(\det(E);x) > 0.$$
        Let $$M := \min\Bigl\{ \min\limits_{1\leq k \leq r-1} \inf\limits_{x\in X} \varepsilon(\xi_k;x) ,\,\  \dfrac{1}{r}  \inf\limits_{x\in X} \varepsilon(\det(E);x)\Bigr\}.$$ The previous observations imply that $M>0.$ We will show that $$\inf_{x\in X}\varepsilon^H(\mathcal{E};x) \geq M >0$$ and this will conclude the proof.

        For a fixed $x\in X$ and fixed $k$ with $1\leq k \leq r-1$, we observe that $$\varepsilon(\xi_k;x) \geq \inf_{x\in X} \varepsilon(\xi_k;x) \geq M.$$
        Similarly, $$\dfrac{1}{r}   \varepsilon(\det(E);x) \geq \dfrac{1}{r}  \inf\limits_{x\in X} \varepsilon(\det(E);x) \geq M.$$
        Thus by the definition of Higgs Seshadri constant, we conclude $$\varepsilon^H(\mathcal{E};x) \geq M > 0$$ for all $x\in X$.
    \end{proof}
\end{thm}

\begin{xrem}\label{upper-bound} \rm
(Upper bound) Using \cite[Proposition 3.14]{FM21}, we have the following upper bound for Higgs Seshadri constants.
Let $\mathcal{E} = (E,\, \theta)$ be a $\text{H}$-nef Higgs bundle of rank $r$ on a smooth projective variety $X$ of dimension $n$. Then for any $x\in X$ we have $$0\leq \varepsilon^H(\mathcal{E};x) = \min\Bigl\{ \min\limits_{1\leq k\leq r-1} \varepsilon(\xi_k;x),\,\  \dfrac{1}{r} \, \varepsilon \bigl( \det(E);x\bigr)\Bigr\} \leq  \min \Bigl\{ \min\limits_{1\leq k\leq r-1} M_k,\,\ \dfrac{1}{r} \sqrt[n]{\det(E)^n} \Bigr\},$$ with  
$$M_k = \sup\limits_W \Big(\dfrac{\xi_k^{\dim W}\cdot [W]}{\binom{\dim W}{\dim \rho_k(W)} \cdot \mult_x \rho_k(W)\cdot \xi_k^{\dim W_x'}[W_x']} \Bigr)^{\dfrac{1}{\dim \rho_k(W)}},$$
where the supremum is taken over  $W$ which ranges through the subvarieties of $\mathcal{G}r_k(\mathcal{E})$ that meet $\rho_k^{-1}(x)$ without being contained in it. In the above,
$W_{x'}$ is a fiber over the flat locus of $W \longrightarrow \rho_k(W).$
\end{xrem}

\begin{thm}\label{thm-res-curve}  (Restriction to curves)
    Let $\mathcal{E} = (E,\, \theta)$ be a Higgs bundle of rank $r$ on a smooth projective variety $X$. Let $x \in X$. Then 
    $$ \varepsilon^H(\mathcal{E}; x) = \inf_{x\in C\subseteq X}\Bigl\{\dfrac{\mu^H_{\min}(\nu^*\mathcal{E})}{\mult_xC}\Bigr\},$$ where for any curve $C\subset X$ passing through $x$, the map $\nu:\widetilde{C}\longrightarrow X$ is the composition of the normalization $\widetilde{C}\rightarrow C$ together with the inclusion $C \hookrightarrow X.$
    \begin{proof}
  Recall that   for any point $x\in X$, we have $$\varepsilon^H(\mathcal{E};x) = \min\Bigl \{ \min\limits_{1\leq k \leq  r-1}\varepsilon(\xi_k;x),\, \dfrac{1}{r}\varepsilon\bigl(\det(E);x\bigr)\Bigr\}.$$
  We will first show that the following two inequalities hold  for any point $x\in X$ and  for any curve $C\subset X$ passing through $x$:
  \begin{enumerate}
      \item $$\inf_{x\in C\subseteq X}\Bigl\{\dfrac{\mu^H_{\min}(\nu^*\mathcal{E})}{\mult_xC}\Bigr\} \leq \min\limits_{1\leq k \leq  r-1}\varepsilon(\xi_k;x).$$
      
\item $$\inf_{x\in C\subseteq X}\Bigl\{\dfrac{\mu^H_{\min}(\nu^*\mathcal{E})}{\mult_xC}\Bigr\} \leq \dfrac{1}{r} \varepsilon\bigl(\det(E);x\bigr).$$
  \end{enumerate}

  Let $C$ be a curve in $X$ passing through the point $x\in X$ and let $y\in \widetilde{C}$ be such that $\nu(y)=x$. Therefore, we have the following fiber product diagram:
        \begin{center}
 \begin{tikzcd} 
 \mathcal{G}r_k(\nu^*\mathcal{E}) \arrow[r, "\widetilde{\nu}"] \arrow[d, "\widetilde{\rho}_k"]
& \mathcal{G}r_k(\mathcal{E}) \arrow[d,"\rho_k"]\\
\widetilde{C}\arrow[r, "\nu" ]
& X
\end{tikzcd}
\end{center}
such that $\widetilde{\nu}^*\xi_k \equiv \mathcal{O}_{ \mathcal{G}r_k(\nu^*\mathcal{E})}(1).$ 
We also observe that $$\mathcal{C}_{\rho_k,x} = \bigcup\limits_{x\in C\subseteq X} \mathcal{C}_{\widetilde{\rho}_k,y}.$$
For any curve $B' = \widetilde{\nu}_*B \in \mathcal{C}_{\rho_k,x}$ for $B\in \mathcal{C}_{{\widetilde{\rho}_k},y}$, let $\mult_y\widetilde{{\rho}_k}_*B =d$ (say). Then $$\mult_y\nu_*\widetilde{\rho_k}_*B = d \mult_y\nu_*\widetilde{C} = d \mult_x C.$$
Thus, by the projection formula,
\begin{align*}
\dfrac{\xi_k\cdot B'}{\mult_x{\rho_k}_*B'} & = \dfrac{\xi_k\cdot\widetilde{\nu}_*B}{\mult_x{\rho_k}_*\widetilde{\nu}_*B} \\
&= \dfrac{\xi_k\cdot\widetilde{\nu}_*B}{\mult_x\nu_*{\widetilde{\rho_k}}_*B} \\
& = \dfrac{\widetilde{\nu}^*\xi_k\cdot B}{\mult_xC\cdot\mult_y\widetilde{\rho_k}_*B} \\
& \geq \dfrac{1}{\mult_xC}\, \varepsilon(\nu^*\xi_k; y)\\
&\geq \dfrac{1}{\mult_xC}\, \varepsilon^H(\widetilde{\nu}^*\mathcal{E};y)\\
& = \dfrac{\mu^H_{\min}(\nu^*\mathcal{E})}{\mult_xC} \geq \inf_{x\in C\subseteq X}\Bigl\{\dfrac{\mu^H_{\min}(\nu^*\mathcal{E})}{\mult_xC}\Bigr\}
\end{align*}
Then $$\inf_{x\in C\subseteq X}\Bigl\{\dfrac{\mu^H_{\min}(\nu^*\mathcal{E})}{\mult_xC}\Bigr\} \leq \min\limits_{i\leq k \leq  r-1}\varepsilon(\xi_k;x).$$\\

This proves (1). Now we also have $\mu_{\min}^H(\nu^*\mathcal{E}) \leq \mu(\nu^*\mathcal{E}).$ Hence 
\begin{align*}
\inf_{x\in C\subseteq X}\Bigl\{\dfrac{\mu^H_{\min}(\nu^*\mathcal{E})}{\mult_xC}\Bigr\} & \leq  \inf_{x\in C\subseteq X}\Bigl\{\dfrac{\mu(\nu^*\mathcal{E})}{\mult_xC}\Bigr\} \\
& = \dfrac{1}{r} \inf_{x\in C\subseteq X}\Bigl\{\dfrac{\det(E)\cdot C}{\mult_xC}\Bigr\}\\
& = \dfrac{1}{r} \varepsilon\bigl(\det(E);x\bigr).
\end{align*}
This proves (2). Combining (1) and (2), we conclude that $$\inf_{x\in C\subseteq X}\Bigl\{\dfrac{\mu^H_{\min}(\nu^*\mathcal{E})}{\mult_xC}\Bigr\} \leq \varepsilon^H(\mathcal{E};x).$$

Now we prove the reverse inequality. Let $C\subset X$ be a curve passing through a point $x\in X$, and  let $B$ be a curve in $\mathcal{C}_{\widetilde{\rho}_k,y}$. Then $\widetilde{\nu}_*B \in \mathcal{C}_{\rho_k,x},$ and by the projection formula
\begin{align*}
    \dfrac{1}{\mult_xC}\Bigl\{\dfrac{\widetilde{\nu}^*\xi_k\cdot B}{\mult_y\widetilde{\rho_k}_*B}\Bigr\}
    & = \dfrac{\xi_k\cdot \widetilde{\nu}_*B}{\mult_x\nu_*\widetilde{\rho_k}_*B}
     = \dfrac{\xi_k\cdot \widetilde{\nu}_*B}{\mult_x{\rho_k}_*\widetilde{\nu}_*B}
    \geq \varepsilon(\xi_k;x)\geq \varepsilon^H(\mathcal{E};x).
\end{align*}
Then $$\dfrac{\widetilde{\nu}^*\xi_k\cdot B}{\mult_y\widetilde{\rho_k}_*B} \geq \mult_xC \cdot \varepsilon^H(\mathcal{E};x).$$
Therefore $$\varepsilon(\widetilde{\nu}^*\xi_k;y) \geq \mult_xC \cdot \varepsilon^H(\mathcal{E};x).$$ Thus we have for each $k$
\begin{align}\label{s14}
\dfrac{1}{\mult_xC}\,\  \varepsilon(\widetilde{\nu}^*\xi_k;y) \geq \varepsilon^H(\mathcal{E};x).
\end{align}Now we will show that $$ \dfrac{1}{r}\,  \varepsilon\bigl(\nu^*\det(E);y\bigr) \geq \mult_xC\cdot \varepsilon^H(\mathcal{E};x),$$
where $\nu(y)=x$.
We note that $$\varepsilon\bigl(\det(E);x\bigr) = \inf\limits_{x\in C\subseteq X} \dfrac{\mu_{\min}(\nu^*\det(E))}{\mult_xC} \leq \dfrac{\mu_{\min}\bigl(\nu^*\det(E)\bigr)}{\mult_xC}.$$
Thus 
\begin{align}\label{s15}
\varepsilon^H(\mathcal{E};x)\leq \dfrac{1}{r}\, \varepsilon\bigl(\det(E);x\bigr) \leq  \dfrac{1}{r} \dfrac{\mu_{\min}\bigl(\nu^*\det(E)\bigr)}{\mult_xC} = \dfrac{1}{r}\dfrac{\varepsilon(\nu^*\det(E);y)}{\mult_xC}.
\end{align}
Hence, combining (\ref{s14}) and (\ref{s15}), we conclude that 
$$ \varepsilon^H(\mathcal{E};x) \leq \dfrac{\varepsilon^H(\nu^*\mathcal{E};y)}{\mult_xC} = \dfrac{\mu_{\min}^H(\nu^*\mathcal{E})}{\mult_xC}.$$

Therefore

$$\varepsilon^H(\mathcal{E};x) \leq \inf\limits_{x\in C\subseteq X} \dfrac{\varepsilon^H(\nu^*\mathcal{E};y)}{\mult_xC} = \inf\limits_{x\in C\subseteq X} \dfrac{\mu_{\min}^H(\nu^*\mathcal{E})}{\mult_xC}.$$
This completes the proof.
    \end{proof}
\end{thm}
\section{Properties of Higgs Seshadri constants}\label{properties}

We continue our study of Higgs Seshadri constants in this section and prove more properties of them. 
\begin{thm}\label{5.1} 
     Let $\mathcal{E}=(E,\theta)$ be a $\text{\rm H}$-nef Higgs bundle of rank $r$ on a smooth projective variety $X$. Then for every point $x\in X$ 
     \begin{align}\label{seq11}
     \varepsilon^H(\mathcal{E};x) \leq \dfrac{1}{r} \varepsilon\bigl(\det(E),x\bigr).
     \end{align}
     Moreover, if $\mathcal{E}$ is curve Higgs semistable, then the equality holds.
     \begin{proof} By definition of the Higgs Seshadri constant for Higgs vector bundles, we have for each point $x\in X$ that \begin{align*}
     \varepsilon^H(\mathcal{E};x) \leq \dfrac{1}{r} \varepsilon\bigl(\det(E),x\bigr).
     \end{align*}
     
     Now note that $\det(E)$ is a nef line bundle as $\mathcal{E}$ is Higgs nef.
         By Theorem \ref{thm-res-curve} $$ \varepsilon^H(\mathcal{E}; x) = \inf_{x\in C\subseteq X}\Bigl\{\dfrac{\mu^H_{\min}(\nu^*\mathcal{E})}{\mult_xC}\Bigr\},$$ for all points $x\in X$, where $\nu$ is as in Theorem \ref{thm-res-curve}. 
         Now, in the case of a curve semistable bundle, we have $\mu_{\min}^H(\nu^*\mathcal{E}) = \mu(\nu^*\mathcal{E})$ for every curve $C$ in $X$ passing through $x$. Thus $$\varepsilon^H(\mathcal{E};x) = \inf_{x\in C\subseteq X}\Bigl\{\dfrac{\mu(\nu^*\mathcal{E})}{\mult_xC}\Bigr\} = \dfrac
         {1}{r} \inf_{x\in C\subseteq X}\dfrac{\det(E)\cdot C}{\mult_xC} = \dfrac{1}{r}\varepsilon(\det(E);x).$$
         Therefore, the equality occurs in (\ref{seq11}).
     \end{proof}
\end{thm}

\begin{exm} \rm
 Recall that the Chern classes of a smooth surface of general type $X$ satisfy the Miyaoka-Yau inequality $$3c_2(X)\geq c_1(X)^2.$$  Let $X$ be a projective minimal surface of general type over an algebraically closed field %$\mathbb{K}$ 
of characteristics zero such that $3c_2(X)=c_1(X)^2.$ Such surfaces exist, eg. ball quotients. Consider the following Simpson's system, i.e., the Higgs bundle $\mathcal{S} = (S,\theta)$ where $S=\Omega^1_X\oplus \mathcal{O}_X,$ and $\theta(\omega,f) = (0,\omega)$. Then the Simpson's system $\mathcal{S}$ is a semistable Higgs bundle with vanishing discriminant (see \cite{L15}), and thus it is curve Higgs semistable. Therefore for every point $x\in X$, we have $$\varepsilon^H(\mathcal{S};x) = \dfrac{1}{\rank(S)}\varepsilon(\det S;x) = \dfrac{1}{3} \varepsilon(K_X;x),$$
where $K_X$ is the canonical divisor on $X$.
    
\end{exm}
\begin{prop}\label{prop6.2}
 Let $\psi:X\longrightarrow Y$ be a map between two smooth projective varieties, and $L$ be a fixed ample line bundle $Y$. Let $\mathcal{E} = (E,\theta)$ be a Higgs semistable Higgs bundle of rank $r$ such that $\Delta(E)\cdot L^{\dim(Y)-2} = 0.$ Then its pullback $\psi^*\mathcal{E}$ is a curve  Higgs semistable Higgs bundle on $X$. 
\begin{proof}
    Let $\phi: C \longrightarrow X$ be a non-constant morphism from a smooth curve $C$ to $X$. If the image of $\phi$ is contained in any fibre of $\psi$, then the underlying vector bundle of the pullback Higgs bundle $\phi^*(\psi^*\mathcal{E})$ is trivial, and hence Higgs semistable. Now let us assume that the image is not contained in any fibre of $\phi$. As $\mathcal{E}$ is Higgs semistable bundle on $Y$ with $\Delta(E)\cdot L^{\dim(Y)-2} =0$,  it is in particular curve Higgs semistable. Thus the pullback Higgs bundle  $(\psi\circ\phi)^*\mathcal{E} = \phi^*(\psi^*\mathcal{E})$ under the non-constant morphism $\psi\circ\phi$  is semistable on $C$. Therefore the pullback Higgs bundle is curve Higgs semistable on $X$.
\end{proof}
\end{prop}
\begin{prop}\label{5.4}
Let $\psi: X = \mathbb{P}_C(W) \longrightarrow C$ be a projective bundle over a smooth curve $C$. Let $\mathcal{E}=(E,\theta)$ be a Higgs semistable Higgs bundle $\mathcal{E}$ on $C$. Then the Higgs bundle $\psi^*\mathcal{E} \otimes L$ is a curve semistable Higgs bundle on $X$ for any line bundle $L$ on $X$.
\begin{proof}
This follows from Proposition \ref{prop6.2} and from the definition of curve Higgs semistability.
\end{proof}
\end{prop}
Recall that a curve semistable Higgs bundle $\mathcal{E} = (E,\theta)$ is Higgs ample if and only if its determinant $\det(E)$ is ample in the usual sense (see \cite[Theorem 3.7]{BMR26}). 
\begin{corl}\label{5.5}
Let $\psi: X = \mathbb{P}_C(W)\longrightarrow C$ be a projective bundle over a smooth curve $C$ and let $\mathcal{V} = (V,\theta)$ be a curve semistable Higgs bundle of rank $r$ on $X$. Further assume that  $c_1(V) \equiv x\xi + yf$, where $\xi$ and $f$ are the numerical classes of
$\mathcal{O}_{\mathbb{P}(W)}(1)$ and a fibre of $\psi$, respectively. Then $\mathcal{V}$  is Higgs ample if and only if $x>0$ and $x\mu_{\min}(W)+y >0$.

In particular, if $\mathcal{E} =(E,\theta)$ is a Higgs semistable Higgs bundle of rank $r$ on $C$, then the Higgs bundle  $\mathcal{V} =\psi^*\mathcal{E} \otimes \mathcal{O}_X(m)$ is Higgs ample if and only if  $m>0$ and $m\mu_{\min}(W)+\mu(\mathcal{E}) > 0.$

 In these cases, the Higgs Seshadri constants 
 are given by 
 $$\varepsilon^H(\mathcal{V};x) = \dfrac{1}{r}\, \varepsilon(\det(V);x)$$
   for every point $x\in \mathbb{P}_C(W)$.
\begin{proof}
Note that by \cite[Lemma 2.1]{Fu11}, the nef cone 
of $\mathbb{P}_C(W)$ is given by
$$\Nef^1(\mathbb{P}_C(W)) = \Bigl\{ x(\xi-\mu_{\min}(W))+yf \mid x,y \in \mathbb{R}_{\geq 0}\Bigr\}.$$  
By applying duality, we obtain that the Mori cone of closed curves 
on $\mathbb{P}_C(W)$ is given by:
$$\overline{\NE}(\mathbb{P}_C(W)) = \Bigl\{ a\bigl(\xi^{p-1} -(\deg(W)-\mu_{\min}(W))\xi^{p-2}f\bigr) + b\xi^{p-2}f\mid  a,b \in \mathbb{R}_{\geq 0}\Bigr\},$$where $p$ is the rank of the vector bundle $W$.
As $\mathcal{V} = (V,\theta)$ is a curve semistable Higgs bundle on $X$,  $\mathcal{V}$ is Higgs ample if and only if its the determinant $\det(V) \equiv x\xi+yf$ is ample if and only if
\begin{enumerate}
    \item $\det(V) \cdot  \bigl(\xi^{p-1} -(\deg(W)-\mu_{\min}(W))\xi^{p-2}f\bigr) = x\mu_{\min}(W)+y > 0,$ and
    \item $\det(V)\cdot \xi^{p-2}f = y >0. $
\end{enumerate}
   Moreover, if $\mathcal{E}$ is a Higgs semistable Higgs bundle on the smooth curve $C$, then the Higgs bundle $\mathcal{V} = \psi^*\mathcal{E} \otimes \mathcal{O}_{\mathbb{P}_C(W)}(m)$ is curve Higgs semistable by Proposition \ref{5.4} with vanishing discriminant for any $m$.  Thus $\mathcal{V} = \psi^*\mathcal{E} \otimes \mathcal{O}_{\mathbb{P}_C(W)}(m)$ is Higgs ample if and only if its determinant $\det(V) \equiv rm\xi+\deg(E)f$ is ample in the usual sense (see \cite[Theorem 3.7]{BMR26}).

   In these cases, by Theorem \ref{5.1} the Higgs Seshadri constants $$\varepsilon^H(\mathcal{V};x) = \dfrac{1}{r}\, \varepsilon(\det(V);x)$$
   for every point $x\in \mathbb{P}_C(W)$.
\end{proof}
\end{corl}
\begin{xrem}\label{5.6}
\rm Let  $\mathcal{E} = (E,\theta)$ be a Higgs semistable Higgs bundle on a smooth curve $C$. Then $\mathcal{E}$ is curve Higgs semistable Higgs bundle by \cite[Lemma 3.3]{BR06}. Thus by \cite[Theorem 3.7]{BMR26} $\mathcal{E}$  is  Higgs ample if and only its determinant line bundle $\det(E)$ is ample if and only if its degree $\deg(E)$ is positive.
\end{xrem}
Following \cite[Chapter 5, Section 2]{H66}, we say that a vector bundle W of rank 2 on a smooth projective curve C is \textit{normalized} if $H^0(C,W) \neq 0$, but $H^0(C,W\otimes L) =0$  for all line bundles $L$ on $C$ with $\deg(L) < 0. $
\begin{prop}\label{5.7}
    Let $\psi: X = \mathbb{P}_C(W)\longrightarrow C$ be a ruled surface defined by a normalized rank 2 bundle on the smooth projective curve $C$ of arbitrary genus $g$ such that $\mu_{\min}(W) = \deg(W).$ Let $i_{\sigma} : \sigma\hookrightarrow X$ and $i_f: f \hookrightarrow X$ are the inclusions of the normalized section $\sigma$ and a fiber $f$ of $\psi$, respectively.  Then 
   for a curve semistable Higgs bundle $\mathcal{E}=(E,\theta)$ of rank $r$ on $X$, the following are equivalent:
   \begin{enumerate}
       \item The Higgs bundle $\mathcal{E}=(E,\theta)$ is Higgs ample.
       \item The Higgs bundles $i_{\sigma}^*\mathcal{E}$ and $i_{f}^*\mathcal{E}$ are Higgs ample.
       \item The degrees $\deg(i^*_\sigma E)$ and $\deg(i_f^* E)$ are positive.
   \end{enumerate}
   
    In this case, by Theorem \ref{5.1} the Higgs Seshadri constants 
    are given by $$\varepsilon^H(\mathcal{E};x) = \dfrac{1}{r} \varepsilon(\det(E);x)
   % = \min\Bigl\{ \mu(E\vert_\sigma),\,\ \mu(E\vert_f)\Bigr\}
    , $$ at every point $x\in X$.
    \begin{proof}
      As $\mathcal{E}$ is curve semistable, the equivalence of (2) and (3) is given by Remark \ref{5.6}. Note that $\mathcal{E}$ is Higgs ample if and only if $x>0$ and $x\mu_{\min}(W)+y = x\deg(W) + y> 0$ if and only if $\deg(i_{\sigma}^*\mathcal{E}) = x\deg(W)+y>0$ and $\deg(i_{f}^*\mathcal{E}) = x > 0$ if and only if  the Higgs bundles $i_{\sigma}^*\mathcal{E}$ and $i_{f}^*\mathcal{E}$ are Higgs ample using Remark \ref{5.6} and Corollary \ref{5.5}.
    \end{proof}
\end{prop}
\begin{exm}\rm
    Let $\psi: X = \mathbb{P}_C(W)\longrightarrow C$ be a ruled surface defined by a normalized rank 2 bundle $W=\mathcal{O}_C\oplus \mathcal{L}$ on the smooth projective curve $C$ of arbitrary genus $g$ where $\deg(\mathcal{L}) <0.$ Then $\mu_{\min}(W) = \deg(W) =\deg(\mathcal{L}).$ Let $\mathcal{E} =(E,\theta)$ be a Higgs semistable Higgs bundle of rank $r$ on $C$, then the Higgs bundle  $\mathcal{V} =\psi^*\mathcal{E} \otimes \mathcal{O}_X(m)$ is Higgs ample if and only if  $m>0$ and $m\mu_{\min}(W)+\mu(E) = m\deg(\mathcal{L})+\mu(E) > 0.$

    For example, a Hirzebruch surface $\mathbb{F}_e = \mathbb{P}_{\mathbb{P}^1}(\mathcal{O}\oplus \mathcal{O}(-e))$ satisfies the hypothesis as above. Then for a Higgs ample bundle of the form $\mathcal{V} =\psi^*\mathcal{E} \otimes \mathcal{O}_X(m)$, we have the Higgs Seshadri constant
    $$\varepsilon^H(\mathcal{V};x) = \dfrac{1}{r} \varepsilon(\det(V);x) = \dfrac{1}{r}\, \varepsilon(rm\xi+\deg(E)f; x) = \varepsilon(m\xi+\mu(E)f; x). $$

    The usual Seshadri constant $\varepsilon(m\xi+\mu(E)f; x)$ can be computed using the results in \cite[Theorem 5.1]{G06}.
\end{exm}
\begin{exm}\rm
    We consider the ruled surface $\psi : X = \mathbb{P}_C(W) \longrightarrow C$ over the elliptic curve $C$ defined by the semistable bundle $W$ which sits in the non-split extension $$0\longrightarrow \mathcal{O}_C \longrightarrow W \longrightarrow \mathcal{O}_C \longrightarrow 0.$$ In this case, $\mu_{\min}(W) = \mu(W) = \deg(W)=0$. Then for a Higgs semistable bundle $\mathcal{E}$ on $C$, the Higgs bundle $\mathcal{V}$ of the form $\mathcal{V} =\psi^*\mathcal{E} \otimes \mathcal{O}_X(m)$ is Higgs ample if and only if the integer $m>0$ and the Higgs semistable bundle $\mathcal{E}$ on $C$ is of positive degree. In this case,
     we have the Higgs Seshadri constant
    $$\varepsilon^H(\mathcal{V};x) = \dfrac{1}{r} \varepsilon(\det(V);x) = \dfrac{1}{r}\, \varepsilon(rm\xi+\deg(E)f; x) = \varepsilon(m\xi+\mu(E)f; x).$$
The usual Seshadri constant $\varepsilon(m\xi+\mu(E)f; x)$ can be computed using the results in \cite[Theorem 6.8]{G06}.
\end{exm}
\begin{prop}
Let $\mathcal{E}=(E,\theta)$ be a Higgs nef Higgs vector bundle of rank $r$ on a smooth projective variety $X$. Let 
\begin{align}
 0=\mathcal{E}_d\subsetneq \mathcal{E}_{d-1}\subsetneq \cdots \subsetneq \mathcal{E}_1\subsetneq \mathcal{E}_0 = \mathcal{E}
\end{align}
be the Higgs Harder-Narasimhan filtration of $\mathcal{E}$ such that the quotient Higgs bundle $\mathcal{Q}_1 = \mathcal{E}_1/ \mathcal{E}_0$  is locally free, and $$\mu^H_{\min}(\nu^*\mathcal{E}) = \mu(\nu^*\mathcal{Q}_1)$$
 for all points $x\in X$, where for any curve $C\subset X$ passing through $x$, the map $\nu:\widetilde{C}\longrightarrow X$ is the composition of the normalization $\widetilde{C}\rightarrow C$ together with the inclusion $C \hookrightarrow X.$ Then for any $x\in X$
 $$\varepsilon^H(\mathcal{E};x) = \dfrac{1}{\rank(Q_1)} \,\ \varepsilon\bigl(\det(Q_1);x\bigr),$$
 where $Q_1$ is the underlying vector bundle of the Higgs quotient bundle $\mathcal{Q}_1$.
 \begin{proof}
By Theorem \ref{thm-res-curve} $$ \varepsilon^H(\mathcal{E}; x) = \inf_{x\in C\subseteq X}\Bigl\{\dfrac{\mu^H_{\min}(\nu^*\mathcal{E})}{\mult_xC}\Bigr\},$$ for all points $x\in X$.

Now by the given hypothesis $\mu^H_{\min}(\nu^*\mathcal{E}) = \mu(\nu^*\mathcal{Q}_1)$, and therefore for every point $x\in X$
$$ \varepsilon^H(\mathcal{E}; x) =  \dfrac{1}{\rank(Q_1)} \inf_{x\in C\subseteq X}\Bigl\{\dfrac{\det(Q_1)\cdot C}{\mult_xC}\Bigr\} = \dfrac{1}{\rank(Q_1)} \,\ \varepsilon\bigl(\det(Q_1);x\bigr).$$
This completes the proof.
 \end{proof}
\end{prop}
\begin{corl}
    Let $\psi : \mathbb{P}_C(W)\longrightarrow C$ be a projective bundle defined by a  rank $r$ vector bundle on a smooth projective curve $C$.  If $\mathcal{E} =(E,\theta)$ is a Higgs bundle of rank $r$ on $C$, then the Higgs bundle  $\mathcal{V} =\psi^*\mathcal{E} \otimes \mathcal{O}_X(m)$ is Higgs ample if and only if  $m>0$ and $m\mu_{\min}(W)+\mu^H_{\min}(\mathcal{E}) > 0.$ In this case, 
    for every point $x\in X$, the Higgs Seshadri constants 
$$ \varepsilon^H(\mathcal{E}; x)  = \dfrac{1}{\rank(Q_1)} \,\ \varepsilon\bigl(\det(Q_1);x\bigr),$$
where $Q_1$ is the underlying vector bundle of the Higgs quotient $\mathcal{Q}_1$ in the Higgs Harder-Narasimhan filtration of $\mathcal{E}$ having minimal slope.
\begin{proof}
   Let \begin{align}
 0=\mathcal{E}_d\subsetneq \mathcal{E}_{d-1}\subsetneq \cdots \subsetneq \mathcal{E}_1\subsetneq \mathcal{E}_0 = \mathcal{E}
\end{align}
be the Higgs Harder-Narasimhan filtration of $\mathcal{E}$.  Then by the uniqueness of the Harder-Narasimhan filtration 
 \begin{align}
 0=\mathcal{V}_d\subsetneq \mathcal{V}_{d-1}\subsetneq \cdots \subsetneq \mathcal{V}_1\subsetneq \mathcal{V}_0 = \mathcal{V}
\end{align}
is the Harder-Narasimhan filtration of $\mathcal{V}$ with successive Higgs semistable Higgs quotients $\mathcal{Q}_i = \Bigl(\psi^*\mathcal{E}_i/\psi^*\mathcal{E}_{i+1}\Bigr) \otimes \mathcal{O}_{\mathbb{P}_C(W)}(m) =\psi^*(\mathcal{E}_i/\mathcal{E}_{i+1}) \otimes \mathcal{O}_{\mathbb{P}_C(W)}(m)$ for $i=0,1,2,\cdots,d-1.$ 

 First assume that  $m>0$ and $m\mu_{\min}(W)+\mu^H_{\min}(\mathcal{E}) > 0.$ Note that, for each $i$, the Higgs quotient $\mathcal{E}_i/\mathcal{E}_{i+1}$ is Higgs semistable, and  we have 
$$\mu(\mathcal{E}_i/\mathcal{E}_{i+1}) \geq \mu(\mathcal{E}_0/\mathcal{E}_{1}) = \mu_{\min}^H(\mathcal{E}) > - m\mu_{\min}(W).$$
Hence each $\mathcal{Q}_i$ is Higgs ample, and thus inductively $\mathcal{V}$ is Higgs ample as extensions of Higgs ample vector bundles are Higgs ample.

Conversely, let $\mathcal{V}$ be Higgs ample, and thus $\mathcal{Q}_1 = \psi^*(\mathcal{E}_0/\mathcal{E}_1) \otimes \mathcal{O}_{\mathbb{P}_C(W)}(m)$ is Higgs ample, since it is a Higgs quotient. Thus we have $m>0$ and $\mu(\mathcal{E}_0/\mathcal{E}_1) = \mu_{\min}^H(\mathcal{E}) > -m\mu_{\min}(W).$

In this case, $\mathcal{Q}_1 = \mathcal{V}_1/ \mathcal{V}_0$  is locally free, and for all points $x\in X$ and for all curve $C$ passing through $x$, we have $$\mu^H_{\min}(\nu^*\mathcal{V}) = \mu(\nu^*\mathcal{Q}_1),$$ where for any curves $C\subset X$ passing through $x$, the map $\nu:\widetilde{C}\longrightarrow X$ is as in Theorem \ref{thm-res-curve}. Then for any $x\in X$
 $$\varepsilon^H(\mathcal{E};x) = \dfrac{1}{\rank(Q_1)} \,\ \varepsilon\bigl(\det(Q_1);x\bigr),$$
 where $Q_1$ is the underlying vector bundle of the Higgs quotient bundle $\mathcal{Q}_1$.
\end{proof}
\end{corl}

\begin{thm}\label{tensor}
    Let $\mathcal{E}=(E,\theta)$ and $\mathcal{F}=(F,\phi)$ be $\text{\rm H}$-nef Higgs bundles on a smooth projective variety $X$. Then $$\varepsilon^H(\mathcal{E}\otimes\mathcal{F};x)  \geq \varepsilon^H(\mathcal{E};x) + \varepsilon^H(\mathcal{F};x)$$
    for all points $x\in X.$  If $X$ is a smooth curve, or if $\mathcal{E} = \mathcal{F}$, then equality holds.
    \begin{proof}
        Recall that for any curve $C\subseteq X$ passing through a point $x$, we have $$ \mu^H_{\min}\bigl(\nu^*(\mathcal{E}\otimes \mathcal{F})\bigr) = \mu^H_{\min}\bigl( \nu^*\mathcal{E} \otimes \nu^*\mathcal{F}\bigr) = \mu^H_{\min}(\nu^*\mathcal{E}) + \mu^H_{\min}(\nu^*\mathcal{F}),$$
        where $\nu : \widetilde{C} \rightarrow X$ is the map as in Theorem \ref{thm-res-curve}.
        Thus, using Theorem \ref{thm-res-curve}  and applying Lemma \ref{lem3.3}(1), we have for all points $x\in X$
        \begin{align*}
        \varepsilon^H(\mathcal{E}\otimes\mathcal{F};x) & = \inf_{x\in C\subseteq X}\Bigl\{\dfrac{\mu^H_{\min}(\nu^*(\mathcal{E}\otimes \mathcal{F}))}{\mult_xC}\Bigr\}\\
        & \geq \inf_{x\in C\subseteq X}\Bigl\{\dfrac{\mu^H_{\min}(\nu^*\mathcal{E})}{\mult_xC}\Bigr\} + \inf_{x\in C\subseteq X}\Bigl\{\dfrac{\mu^H_{\min}(\nu^*\mathcal{F})}{\mult_xC}\Bigr\}\\
        & = \varepsilon^H(\mathcal{E};x) + \varepsilon^H(\mathcal{F};x).
        \end{align*}
        Now, if $X$ is a smooth curve, then for all points $x\in X$, we have 
        \begin{align*}
        \varepsilon^H(\mathcal{E}\otimes\mathcal{F};x) = \mu^H_{\min}(\mathcal{E}\otimes \mathcal{F}) = \mu^H_{\min}(\mathcal{E})+\mu^H_{\min}(\mathcal{F}) =  \varepsilon^H(\mathcal{E};x) +  \varepsilon^H(\mathcal{F};x).
        \end{align*}
        Moreover, if $X$ is a smooth projective variety of any dimension $d \geq 1$, and $\mathcal{E} = \mathcal{F}$, then for all points $x\in X$, we have
        $$ \varepsilon^H(\mathcal{E}\otimes\mathcal{F};x)  = \inf_{x\in C\subseteq X}\Bigl\{\dfrac{\mu^H_{\min}(\nu^*(\mathcal{E}\otimes \mathcal{E}))}{\mult_xC}\Bigr\} = 2  \inf_{x\in C\subseteq X}\Bigl\{\dfrac{\mu^H_{\min}(\nu^*(\mathcal{E}))}{\mult_xC}\Bigr\} = 2 \, \varepsilon^H(\mathcal{E}; x)$$

    This completes the proof. \end{proof}
\end{thm}
\begin{thm}\label{quotient}
   Let $\mathcal{E} = (E,\, \theta)$ be a $\text{H}$-nef Higgs bundle of rank $r$ on a smooth projective variety $X$, and $\mathcal{F}=(F,\overline{\theta}\vert_F)$ be a $\theta$-invariant quotient Higgs bundle of $\mathcal{E}.$ Then for any point $x\in X,$ we have 
   $$\varepsilon^H(\mathcal{F};x) \geq \varepsilon^H(\mathcal{E};x).$$
   \begin{proof}
     Note that  $\mathcal{F}$ is also a $\text{H}$-nef Higgs quotient vector bundle of $\mathcal{E}$. Then, for each $k$ with $1\,\leq\, k \,\leq\, \rank(\mathcal{F})-1$, there exists a morphism $i\,:\,\mathcal{G}r_k(\mathcal{F}) \,\longrightarrow\, \mathcal{G}r_k(\mathcal{E})$ over $X$ such that $i^*\mathcal{Q}_{\mathcal{E},k} \,=\, \mathcal{Q}_{\mathcal{F},k}$. Thus $\det\bigl( i^*\mathcal{Q}_{\mathcal{E},k} \bigr) \,=\, i^*\mathcal{O}_{\mathcal{G}r_k(\mathcal{E})}(1) \,=\, \det\bigl( \mathcal{Q}_{\mathcal{F},k} \bigr) \,=\, \mathcal{O}_{\mathcal{G}r_k(\mathcal{F})}(1).$

       Let $\xi_{\mathcal{F},k} \equiv \mathcal{O}_{\mathcal{G}r_k(\mathcal{F})}(1)$ and $\xi_{\mathcal{E},k} \equiv \mathcal{O}_{\mathcal{G}r_k(\mathcal{E})}(1).$
By \cite[Corollary 3.18]{FM21} we have 
       $$\varepsilon(\xi_{\mathcal{F},k};x) \geq \varepsilon(\xi_{\mathcal{E},k};x) \geq \varepsilon^H(\mathcal{E};x).$$

       Recall that,  by Theorem \ref{thm-res-curve}, $$\varepsilon^H(\mathcal{E};x) = \inf\limits_{x\in C\subseteq X} \Bigl\{\dfrac{\mu_{\min}^H(\nu^*\mathcal{E})}{\mult_xC}\Bigr\},$$
       where the maps $\nu$ are as in Theorem \ref{thm-res-curve}. 

       Note that for any curve $C\subseteq X$ passing through a point $x$, we have $$\mu_{\min}^H(\nu^*\mathcal{E}) \leq \mu(\nu^*\mathcal{F}),$$ as we have the following quotient map $$ \nu^*\mathcal{E} \longrightarrow \nu^*\mathcal{F} \longrightarrow 0.$$ 
       This, in particular, implies that $$\varepsilon^H(\mathcal{E};x) = \inf\limits_{x\in C\subseteq X} \Bigl\{\dfrac{\mu_{\min}^H(\nu^*\mathcal{E})}{\mult_xC}\Bigr\} \leq \inf\limits_{x\in C\subseteq X} \Bigl\{ \dfrac{ \mu(\nu^*\mathcal{F})}{\mult_xC}\Bigr\} = \dfrac{1}{\rank(\mathcal{F})}\,\ \varepsilon\bigl(\det(F);x\bigr).$$
        This shows that $$\varepsilon^H(\mathcal{F};x) \geq \varepsilon^H(\mathcal{E};x),$$ as required. 
   \end{proof}
\end{thm}

\begin{corl}\label{exact}
    Let $$0\,\longrightarrow\, \mathcal{E}_1\,\longrightarrow\, \mathcal{E} \,\longrightarrow\, \mathcal{E}_2 \,\longrightarrow\, 0$$ be a short exact sequence of Higgs vector bundles over a smooth projective variety $X$, where $\mathcal{E}_1\,=\,(E_1,\,\theta_1)$ is an invariant Higgs subbundle of $\mathcal{E} = (E,\theta)$ and 
    $\mathcal{E}_2\,=\,(E_2,\,\theta_2)$ is the corresponding Higgs quotient. Suppose that $\mathcal{E}_1$ and $\mathcal{E}_2$ are Higgs nef. Then $$\varepsilon^H(\mathcal{E};x) \geq \min\Bigl\{ \varepsilon^H(\mathcal{E}_1;x),\, \varepsilon^H(\mathcal{E}_2;x)\Bigr\}$$ for all $x\in X$.  In particular, if $\varepsilon^H(\mathcal{E}_1;x) \geq \varepsilon^H(\mathcal{E}_2;x)$, then $\varepsilon^H(\mathcal{E};x) = \varepsilon^H(\mathcal{E}_2;x)$.
    
    Moreover, if $\mathcal{E} = \mathcal{E}_1\oplus \mathcal{E}_2$ for two $\theta$-invariant Higgs subbundles $\mathcal{E}_1$ and $\mathcal{E}_2$ of $\mathcal{E},$ then $$\varepsilon^H(\mathcal{E};x) =\min\Bigl\{ \varepsilon^H(\mathcal{E}_1;x),\, \varepsilon^H(\mathcal{E}_2;x)\Bigr\}$$ for all $x\in X$. 
\end{corl}
\begin{proof} By using Theorem \ref{thm-res-curve}, we assume without loss of generality that $X$ is a smooth curve. 

Therefore using Lemma \ref{lem2.2} we conclude that $$\varepsilon^H(\mathcal{E};x) \geq \min\Bigl\{ \varepsilon^H(\mathcal{E}_1;x),\, \varepsilon^H(\mathcal{E}_2;x)\Bigr\}$$ for all $x\in X$.

Now if $\varepsilon^H(\mathcal{E}_1;x) \geq \varepsilon^H(\mathcal{E}_2;x)$ for all $x\in X$, then $$\varepsilon^H(\mathcal{E};x) \geq \min\Bigl\{ \varepsilon^H(\mathcal{E}_1;x),\, \varepsilon^H(\mathcal{E}_2;x)\Bigr\} = \varepsilon^H(\mathcal{E}_2;x)$$ for all points $x\in X$. On the other hand, we also have, using Theorem \ref{quotient}, that $$\varepsilon^H(\mathcal{E}_2;x) \geq \varepsilon^H(\mathcal{E};x)$$ for all $x\in X$. This proves the equality in this case. 

In particular,  we have, by the above observation, that
$$\varepsilon^H(\mathcal{E}_1\oplus \mathcal{E}_2; x) =\min\Bigl\{ \varepsilon^H(\mathcal{E}_1;x),\, \varepsilon^H(\mathcal{E}_2;x)\Bigr\}$$ for all $x\in X$. 
\end{proof}
\begin{thm}\label{symmetric}
     Let $\mathcal{E}=(E,\theta)$  be a $\text{\rm H}$-nef Higgs bundle on a smooth projective variety $X$. Then for any point $x\in X$ $$\varepsilon^H(\Sym^m\mathcal{E};x) = m\,\ \varepsilon^H(\mathcal{E};x).$$
     \begin{proof}
         By Theorem \ref{thm-res-curve} and applying Lemma \ref{lem3.3}(3), we have for any point $x\in X$ 
         $$\varepsilon^H(\Sym^m\mathcal{E};x) = \inf\limits_{x\in C\subseteq X}\Bigl\{ \dfrac{\mu_{\min}^H\bigl( \nu^*(\Sym^m\mathcal{E})\bigr)}{\mult_xC} \Bigr\} = m \inf\limits_{x\in C\subseteq X}\Bigl\{ \dfrac{\mu_{\min}^H\bigl( \nu^*\mathcal{E}\bigr)}{\mult_xC} \Bigr\} = m\,\ \varepsilon^H(\mathcal{E};x).$$
     \end{proof}
\end{thm}

\begin{lem}\label{pullback}
Let $f:X\longrightarrow Y$ be a surjective morphism between two smooth projective varieties $X$ and $Y$ and let $\mathcal{E}$ be a Higgs nef Higgs bundle of rank $r$ on $Y$. Fix a point $x\in X$ and let $f(x) = y$. Then 
$$\varepsilon^H(f^*\mathcal{E};x) \geq \varepsilon^H(\mathcal{E};f(x)).$$

\begin{proof}
    Consider the following cartesian square
    \begin{center}
 \begin{tikzcd} 
 \mathcal{G}r_k(f^*\mathcal{E}) \arrow[r, "\widetilde{f}"] \arrow[d, "\widetilde{{\rho}_k}"]
& \mathcal{G}r_k(\mathcal{E}) \arrow[d,"\rho_k"]\\
X\arrow[r, "f" ]
& Y
\end{tikzcd}
\end{center}
such that $\mathcal{O}_{\mathcal{G}r_k(f^*\mathcal{E})}(1) \equiv \widetilde{f}^*\xi_k$, where $\xi_k \equiv \mathcal{O}_{\mathcal{G}r_k(\mathcal{E})}(1)$ for each $1\leq k \leq r-1.$

Recall that $$\varepsilon^H(f^*\mathcal{E};x) = \min \Bigl\{ \min\limits_{1\leq k \leq r-1} \varepsilon(\widetilde{f}^*\xi_k;x), \,\ \dfrac{1}{r}\, \varepsilon\bigl( f^*\det(E); x\bigr) \Bigr\}.$$
Note that for any $k$ with $1\leq k \leq r-1$, we have $$\varepsilon(\widetilde{f}^*\xi_k;x) \geq \varepsilon(\xi_k;y) \geq \varepsilon^H(\mathcal{E};y).$$
By the projection formula $$\dfrac{f^*\det(E)\cdot C}{\mult_xC} = \dfrac{\det(E)\cdot f_*C}{\mult_xC} \geq \dfrac{\det(E)\cdot f_*C}{\mult_{f(x)}f_*C} \geq \varepsilon(\det(E); y). $$
This shows that  $$\dfrac{1}{r}\, \varepsilon(f^*\det(E);x) \geq \dfrac{1}{r} \varepsilon(\det(E); y) \geq \varepsilon^H(\mathcal{E};y).$$
Therefore by definition, we conclude that $$\varepsilon^H(f^*\mathcal{E};x) \geq \varepsilon^H(\mathcal{E};f(x)).$$
\end{proof}
\end{lem}

\begin{lem}\label{box-prod}
    Let $\mathcal{E}$ and $\mathcal{F}$ be Higgs two Higgs nef Higgs vector bundles on smooth projective varieties $X$ and $Y$, respectively such that either $\mathcal{E}$ or $\mathcal{F}$ is Higgs ample. Consider the projection maps $p_1: X\times Y\longrightarrow X$ and $p_2:X\times Y \longrightarrow Y$. Then $p_1^*\mathcal{E} \otimes p_2^* \mathcal{F}$ is Higgs ample.
    \begin{proof}
   Without loss of generality, we assume that $\mathcal{E}$ is Higgs ample. We first show that $p_1^*\mathcal{E} \otimes p_2^* \mathcal{F}$ is Higgs nef using the Barton-Kleiman criterion \cite[Lemma 3.3]{BBG19} for Higgs nefness. Let $C$ be a smooth curve and $f:C\longrightarrow X\times Y$ be a non-constant morphism. Let $\pi_1 = p_1\circ f$ and $\pi_2 = p_2 \circ f.$ Then by the Barton-Kleiman criterion for Higgs nefness we have $$\mu_{\min}^H(\pi_1^*\mathcal{E}) \geq 0\,\ \text{and}\,\  \mu_{\min}^H(\pi_2^*\mathcal{F}) \geq 0.$$ 

   Now $$\mu_{\min}^H(f^*(p_1^*\mathcal{E}\otimes p_2^*\mathcal{F})) \geq \mu_{\min}^H(\pi_1^*\mathcal{E}) + \mu_{\min}^H(\pi_2^*\mathcal{F}) \geq  0.$$
   \vspace{3mm}   
   This shows that $p_1^*\mathcal{E} \otimes p_2^* \mathcal{F}$ is Higgs nef. To prove that $p_1^*\mathcal{E} \otimes p_2^* \mathcal{F}$ is Higgs ample, by Theorem 
   \ref{ses-criteria},
   it is now enough to show that $$\inf\limits_{(x,y)\in X \times Y} \varepsilon^H(p_1^*\mathcal{E} \otimes p_2^* \mathcal{F}; (x,y)) > 0.$$

   We have 
   \begin{align*}
       \inf\limits_{(x,y)\in X \times Y} \varepsilon^H(p_1^*\mathcal{E} \otimes p_2^* \mathcal{F}; (x,y)) & \geq \inf\limits_{(x,y)\in X \times Y} \varepsilon^H(p_1^*\mathcal{E};(x,y)) + \inf\limits_{(x,y)\in X \times Y} \varepsilon^H(p_2^*\mathcal{F};(x,y))\\
       & \geq \inf\limits_{x\in X} \varepsilon^H(\mathcal{E};x) \,\ + \,\ \inf\limits_{y \in  Y}\varepsilon^H(\mathcal{F};y)  \\
       & > 0+0 =0  \,\ \,\ (\text{as  $\mathcal{E}$ is Higgs ample and $\mathcal{F}$ is  Higgs nef}).
   \end{align*} 
    This completes the proof.
    \end{proof}
\end{lem}
\begin{xrem}\label{xrem10}
\rm Let $\mathcal{E}$ be a Higgs nef Higgs vector bundle on a smooth projective curve $C$. Then by Theorem 
\ref{thm-curves}$$\varepsilon^H(\mathcal{E};x) = \mu^H_{\min}(\mathcal{E}) $$ for every point $x\in C.$ 
%Thus for all points $x\in C$
%\begin{align}\label{s16}
%\varepsilon^H(\mathcal{E};x) \geq \dfrac{1}{\rank({\mathcal{E}})}.
%\end{align}

Now fix any positive real number $\delta$ and a positive integer $n$. Let $r$ be a positive integer such that $0 < \dfrac{1}{r} < \delta$. Note that for every $r\geq 1$, there exists a stable vector bundle $V_r$ of rank $r\geq 2$ and degree 1 on any smooth curve $C$ with genus $g(C) \geq 2$  ( see \cite[Lemma 3.1]{Hac00}). Thus the Higgs bundle $\mathcal{V}_r = (V_r,\theta)$ is Higgs stable and Higgs ample (by applying Remark \ref{5.6}) for any Higgs field $\theta.$ In this case,  $$\varepsilon^H(\mathcal{V}_r;x) = \mu^H_{\min}(\mathcal{V}_r) = \dfrac{1}{\rank(\mathcal{V}_r)}$$ for all points $x\in C$. 
 Fix a smooth irreducible projective variety $Y$ of dimension $n - 1 \geq 0$ and an ample line bundle $L$ on $Y$. Consider the product $X = C \times Y$ together with
the projection maps $$p_1:X\longrightarrow C \hspace{3mm} \text{and} \hspace{3mm} p_2:X\longrightarrow Y.$$
Let $x = (c, y) \in X$ be an arbitrary point of $X$. Then $\mathcal{E}_r := p_1^*\mathcal{V}_r\otimes p_2^*L $ is Higgs ample by Lemma \ref{box-prod}.

Now consider the smooth curve $B = C\times \{y\} \xrightarrow{i} C\times Y$. Then 
$$ \varepsilon^H(\mathcal{E}_r;x) \leq \mu^H_{\min}(i^*\mathcal{E}_r) = \mu(\mathcal{V}_r) = \dfrac{1}{\rank(\mathcal{V}_r)} = \dfrac{1}{r} < \delta. $$
This shows that for the chosen $\delta>0$, there is a smooth projective
variety $X$ of dimension $n$ and a Higgs  vector bundle $\mathcal{E}_r$  on $X$ which is Higgs ample  such that
$$ \varepsilon^H(\mathcal{E}_r;x) < \delta$$ for all points $x\in X$.

Note that the underlying vector bundle $E_r=p_1^*V_r\otimes p_2^*L$ of $\mathcal{E}_r$  is curve semistable, and thus $\mathcal{E}_r$ is curve Higgs semistable for any Higgs field $\theta$ on $V_r$. Therefore, using \cite[Lemma 3.28]{FM21} and applying Theorem \ref{5.1}, we have for any point $x\in X$ $$\dfrac{1}{r}\varepsilon(\det(E_r);x) = \varepsilon(E_r;x) = \varepsilon(\mathcal{E}_r;x) < \delta.$$
\end{xrem}
\begin{thm}\label{miranda-exm}
    (A Miranda type result)  For any positive real number $\delta$ and any positive integer $n$, there exists a smooth projective variety $X$ of dimension $n$ and a Higgs ample Higgs bundle $\mathcal{E} = (E,\theta)$ on $X$ such that $$\varepsilon(E;x) < \varepsilon^H(\mathcal{E};x) < \delta$$ for all $x\in X$.  
\begin{proof}
     Choose a positive integer $r$ such that $0 < \dfrac{1}{r} < \delta$. Let $C$ be a smooth projective curve of genus $g=2$ and let $L$ be a line bundle of degree 1 on $C$. Consider the vector bundle $E_r=L\oplus \mathcal{O}_C^{\oplus r-1}$, and $\theta : L \longrightarrow \mathcal{O}_C^{\oplus r-1} \otimes \Omega^1_C$ and $\theta : \mathcal{O}_C \longrightarrow \mathcal{O}_C\otimes \Omega_C^1$. Then the vector bundle $E_r$ is unstable as $L$ is a destabilizing line subbundle of $E_r$, but $\mathcal{E}_r=(E_r,\theta)$ is a nilpotent Higgs ample Higgs semistable Higgs bundle on $C$ by Remark \ref{5.6}.

     Let $Y$ be a smooth projective variety of dimension $n-1$ and $L$ be a very ample line bundle on $Y$. Consider the product $X=C\times Y$ and the Higgs bundle $\mathcal{V}_r = p_1^*\mathcal{E}_r\otimes p_2^*L$ on $X$. Note that $\mathcal{V}_r$ is a curve Higgs semistable bundle on $X$, but the underlying vector bundle $V_r=p_1^*E_r\otimes p_2^*L$ is nef and not ample, and it is not curve semistable.

     In this case, $$\varepsilon(V_r;(c,y)) < \varepsilon^H(\mathcal{V}_r;(c,y)) = \dfrac{1}{r}\varepsilon(\det(V_r);(c,y))$$ for every point $(c,y) \in X$.

     For any point $(c,y) \in X$, consider the smooth curve $B=C\times \{y\} \xrightarrow{i} C\times Y.$ Then 
     $$\varepsilon^H(\mathcal{V}_r;x) \leq \mu_{\min}^H(i^*\mathcal{V}_r) = \dfrac{1}{\rank(E_r)} = \dfrac{1}{r} < \delta.$$
This completes the proof. 
\end{proof}
\end{thm}
\section*{Acknowledgements}
We thank Ugo Bruzzo for useful comments on a preliminary version of this article. The first author is supported by 
a grant from Infosys Foundation and 
ANRF MATRICS grant (ANRF/ARGM/2025/002374/MTR). The second author is supported by FRS grant from IIT (ISM) Dhanbad with FRS Project No. MISC 0267. The third author is supported by  Initiation  Research Grant(IRG) from IIIT-Delhi.

\end{document}